\documentclass[a4paper,10pt,reqno]{amsart}
\usepackage{amsmath}
\usepackage{amsthm, enumerate}

\usepackage{graphicx}
\usepackage{amssymb}

\usepackage{appendix}

\usepackage{color}

\usepackage{url}

\usepackage[italian,english]{babel}

\usepackage{fancyhdr}

\usepackage{calc}

\usepackage[text={6in,8.6in},centering]{geometry}

\usepackage{srcltx}

\makeatletter
\def\cleardoublepage{\clearpage\if@twoside \ifodd\c@page\else%
         \hbox{}%
     \thispagestyle{empty}
     \newpage%
     \if@twocolumn\hbox{}\newpage\fi\fi\fi}
\makeatother

\let\cleardoublepage\clearpage

\addto\captionsitalian{}

\newtheorem{thm}{Theorem}[section]

\newtheorem{lem}[thm]{Lemma}
\newtheorem{pro}[thm]{Proposition}
\newtheorem{den}[thm]{Definition}
\newtheorem{oss}[thm]{Remark}
\numberwithin{equation}{section}

\newcommand{\D}[1]{\mbox{\rm #1}}
\newcommand{\dd}{\D{d}}

\begin{document}

\title[Asymptotic behaviour for the fractional PME with variable density]{On the asymptotic behaviour of solutions to the \\ fractional porous medium equation with variable density}
\author {Gabriele Grillo, Matteo Muratori, Fabio Punzo}

\address {Gabriele Grillo, Matteo Muratori: Dipartimento di Matematica ``F. Brioschi'', Politecnico di Milano, Piaz\-za Leonardo da Vinci 32, 20133 Milano, Italy}

\email {gabriele.grillo@polimi.it}
\email{matteo.muratori@polimi.it}

\address{Fabio Punzo: Dipartimento di Matematica ``F. Enriques'', Universit\`a degli Studi di Milano, via Cesare Saldini 50, 20133 Milano, Italy}

\email{fabio.punzo@unimi.it}

\begin{abstract}
We are concerned with the long time behaviour of solutions to the
fractional porous medium equation with a variable spatial density. We prove
that if the density decays slowly at infinity, then the solution
approaches the Barenblatt-type solution of a proper singular
fractional problem. If, on the contrary, the density decays rapidly
at infinity, we show that the minimal solution multiplied by a suitable
power of the time variable converges to the minimal solution of a certain fractional sublinear elliptic equation.
\end{abstract}
\maketitle
\tableofcontents

\section{Introduction}
We investigate the asymptotic behaviour, as $t\to\infty,$ of nonnegative solutions to the following parabolic nonlinear, degenerate, \emph{nonlocal} weighted problem:
\begin{equation}\label{eq: regular}
\begin{cases}
\rho(x) u_t + (-\Delta)^s ( u^m ) = 0  &  \ \textrm{in }  \mathbb{R}^d \times (0,\infty) \, , \\
u = u_0  & \ \textrm{on }  \mathbb{R}^d \times \{ 0 \}  \, ,
\end{cases}
\end{equation}
where the initial datum $u_0$ is nonnegative and belongs to
$$ L^1_\rho({\mathbb R}^d)=\left\{u: \| u \|_{1,\rho}=\int_{\mathbb{R}^d}|u(x)| \, \rho(x){\rm d}x<\infty \right\} $$
and the weight $\rho$ is assumed to be positive, locally essentially bounded away from zero (namely $ \rho^{-1}\in L^{\infty}_{\rm loc}({\mathbb R}^d) $) and to satisfy suitable decay conditions at infinity, which we shall specify later.
As for the parameters involved, we shall assume throughout the paper that $m > 1 $ and $d > 2s$. Moreover, for all $s \in (0,1)$, the symbol $ (-\Delta)^s$ denotes the fractional Laplacian operator, that is
\begin{equation}\label{eq: def-basilare-frac-lap}
(-\Delta)^s(\phi)(x)= p.v.\ C_{d,s} \int_{\mathbb{R}^d}  \frac{\phi(x)-\phi(y)}{|x-y|^{d+2s}} \, \mathrm{d}y \quad \forall x \in \mathbb{R}^d \, , \ \forall \phi \in C^\infty_c(\mathbb{R}^d) \, ,
\end{equation}
$ C_{d,s} $ being a suitable positive constant depending only on $s$ and $d$. For less regular functions, the fractional Laplacian is meant in the usual distributional sense.

For weights $\rho(x) $ that decay \emph{slowly} as $|x|\to\infty $, we shall also be able to consider the more general problem
\begin{equation}\label{eq: measure}
\begin{cases}
\rho(x) u_t + (-\Delta)^s ( u^m ) = 0  &  \ \textrm{in} \ \mathbb{R}^d \times (0,\infty) \, , \\
\rho(x) u = \mu & \ \textrm{on }  \mathbb{R}^d \times \{ 0 \} \, ,
\end{cases}
\end{equation}
where $\mu$ is a positive finite measure. More precisely, here we shall assume that $\rho $ complies with the following assumptions:
\begin{equation*}\label{eq: ass-rho}
c |x|^{- \gamma_0 } \le \rho(x) \le C |x|^{- \gamma_0 } \ \
\textrm{a.e.\ in } B_1 \quad \textrm{and} \quad  c |x|^{-\gamma}
\le \rho(x) \le C |x|^{-\gamma} \ \ \textrm{a.e.\ in } B_1^c
\end{equation*}
for some $ \gamma \in [0,2s)$, $ \gamma_0 \in [0,\gamma] $ and $0< c
< C $ ($B_R$ denotes the ball of radius $R$ centered at $x=0$,
while $B_R^c$ denotes its complement). Note that in this case
$\rho(x) $ is allowed to have a singularity at $x = 0 $.

\smallskip
The local version of problem \eqref{eq: regular}, that is
\begin{equation}\label{e3i}
\begin{cases}
 \rho(x) u_t -\Delta( u^m ) = 0  &  \,\,  \textrm{in} \ \mathbb{R}^d \times (0,\infty)\, , \\
 u = u_0 & \ \textrm{on} \ \mathbb{R}^d\times\{0\}  \, ,
\end{cases}
\end{equation}
has been largely studied in the literature (see e.g.\
\cite{KR2,Eid90,EK,KKT,RV08,Pu1,GMPo,GMP}). In particular, for
$d\geq 3$, it is shown that \eqref{e3i} admits a unique very weak
solution if $\rho(x)$ decays \emph{slowly} as $|x|\to\infty$,
while nonuniqueness prevails when $\rho(x)$ decays \emph{fast
enough} as $|x|\to\infty$. In the latter case, uniqueness can be
restored by imposing extra conditions at infinity on the solutions. Also note that,
independently of the behaviour of $\rho(x)$ as $|x|\to\infty$,
existence and uniqueness of the so-called {\em weak energy
solutions} (namely solutions belonging to suitable functional
spaces) hold true (see \cite{GMPo}). Furthermore, the long time
behaviour of solutions to problem \eqref{e3i} has been addressed
in \cite{RV06, RV09} and \cite{KRV10}. To be specific, in
\cite{RV09} it is proved that if $\|u_0\|_{1,\rho}=M>0$, $\rho>0$
and $\rho(x)\sim |x|^{-\gamma}$ as $|x|\to \infty$, for some
$\gamma\in [0,2)$, then the solution $u$ to problem \eqref{e3i}
satisfies
\[ \lim_{t \to \infty} \left\| u(t)- u^*_M(t) \right\|_{1,\rho} =0  \]
and
\[ \lim_{t \to \infty} t^{\alpha} \left\|u(t)-u^*_M(t) \right\|_{\infty} =0 \, . \]
Here, $u^*_M$ is the self-similar Barenblatt solution of mass $\int_{\mathbb{R}^d} u^*_M \rho =  M$, that is
\[u^*_M(x,t)=t^{-\alpha}F\left(t^{-\kappa}|x|\right) \quad  \forall (x,t) \in \mathbb{R}^d \times (0,\infty)  \, ,  \]
with
$$ F(\xi)=(C-k\xi^{2-\gamma})_+^{\frac 1{m-1}} \quad \forall \xi\geq 0  $$
for suitable positive constants $C$ and $k$ depending on $M$, $m$, $d$, $\gamma$. Moreover,
\[ \quad \alpha=(d-\gamma)\kappa \, , \quad \kappa=\frac 1{d(m-1)+2-m\gamma}  \, . \]
We stress that $u^*_M$ solves the singular problem
\[
\begin{cases}
|x|^{-\gamma} u_t -\Delta( u^m ) = 0  &  \quad  \textrm{in } {\mathbb R}^d\times(0,\infty) \, , \\
|x|^{-\gamma} u  = M \delta &  \quad \textrm{on } \mathbb{R}^d \times \{ 0 \} \, ,
\end{cases}
\]
where $M=\|u_0\|_{1,\rho}$ and $ \delta$ is the {\it Dirac delta}
centred at $x=0 $.
Note that, for $\rho \equiv 1$ the same asymptotic results were shown in \cite{FK} and in \cite{Vaz03}.

\smallskip

On the contrary, in \cite{KRV10} it is proved that if $\rho>0$ and $\rho(x)\sim
|x|^{-\gamma}$ as $|x|\to\infty$, for some $\gamma>2$, then the minimal solution to problem \eqref{e3i}, which is unique in the class of solutions fulfilling
$$
\frac{1}{R^{d-1}}\int_{\partial B_R}\int_0^t u^m(x,\tau)\,{\rm d}\tau{\rm d}S \to 0 \quad \textrm{as } R\to\infty
$$
for all $t>0$, satisfies
\[t^{\frac 1 {m-1}}u(x,t)\to (m-1)^{-\frac{1}{m-1}} W^{\frac 1m}(x)\quad \textrm{as } t\to \infty \, , \quad \textrm{uniformly w.r.t.\ } x \in \mathbb{R}^d \,. \]
Here $W$ is the unique (minimal) positive solution to the sublinear elliptic equation
\[-\Delta W = \rho \, W^{\frac 1m} \quad \textrm{in } \mathbb{R}^d \, ,  \]
and it is such that
\[ \lim_{|x|\to\infty} W(x)=0 \,. \]

\smallskip

Problem \eqref{eq: regular} with $\rho\equiv 1$, nonnegative initial data $u_0$ in $L^1({\mathbb R}^d)$ and $s\in (0,1)$, namely
\begin{equation}\label{e1i}
\begin{cases}
u_t + (-\Delta)^s ( u^m ) = 0      & {\rm in} \ {\mathbb R}^d\times(0,\infty)  \, , \\
u = u_0  &  {\rm on} \ {\mathbb R}^d \times \{ 0 \}  \, ,
\end{cases}
\end{equation}
has been addressed in the breakthrough papers
\cite{DQRV2, DQRV} for $m>0$. In particular, existence,
uniqueness and qualitative properties of solutions are
studied. Moreover, in \cite{BV} sharp quantitative a priori estimates for
solutions are proved. The asymptotic behaviour
has recently been investigated in \cite{VazConstr}. More precisely, it is
first shown that, for any $M>0$, there exists a unique solution
$u^*_M$ to the singular problem
\[
\begin{cases}
u_t + (-\Delta)^s ( u^m ) = 0  &{\rm in}\ {\mathbb R}^d\times(0,\infty) \, , \\
u = M \delta & {\rm on} \ {\mathbb R}^d \times \{ 0 \}  \, .
\end{cases}
\]
Furthermore, such $u^*_M$ has the following self-similar form:
\[u^*_M(x,t)=t^{-\alpha}f(t^{-\kappa}|x|) \quad \forall(x,t) \in \mathbb{R}^d \times (0,\infty) \, , \]
where
$$
\alpha=\frac{d}{d(m-1)+2s} \, , \quad \kappa=\frac 1{d(m-1)+2s}
$$
and the profile $f:[0,\infty)\to (0,\infty)$ is a bounded,
H\"older continuous decreasing function, with $f(r)\to 0$ as $r\to
\infty$. In view of such properties, $u^*_M$ is still called a {\it
Barenblatt-type} solution. Then it is proved that the solution $u$ to problem \eqref{e1i} satisfies
\[
\lim_{t\to \infty} \left\| u(t)- u^*_M(t) \right\|_1 = 0
\]
and
\begin{equation}\label{Linfty}
\lim_{t\to \infty}  t^{\alpha} \left\| u(t)- u^*_M(t) \right\|_\infty = 0 \, .
\end{equation}

\smallskip

Existence and uniqueness of nonnegative bounded solutions to problem
\eqref{eq: regular} for nonnegative initial data $u_0 \in L^1_\rho({\mathbb R}^d)\cap L^\infty({\mathbb R}^d)$ and strictly
positive weights have been investigated in \cite{PT1, PT2}. More precisely, it is proved that if $\gamma\in (0, 2s)$ and there exists
$C_0>0$ such that if
\[
\rho(x)\geq C_0 |x|^{-\gamma} \quad \textrm{a.e.\ in } B_1^c \, ,
\]
then problem \eqref{eq: regular} admits a unique bounded solution. Furthermore, when
$\gamma\in (2s, \infty)$ and there exists $C_0>0$ such that
\begin{equation}\label{decay}
\rho(x)\leq C_0 |x|^{-\gamma}\quad \textrm{a.e.\ in } B_1^c \, ,
\end{equation}
we have existence of solutions satisfying a proper decaying condition at infinity. In the present paper we shall prove uniqueness within a certain class of solutions under the weaker requirement that \eqref{decay} holds true with  $d>4s$ and $\gamma\in (2s, d-2s]\cap(4s,\infty)$ (see Theorem \ref{prthm6}). Actually, for generic positive densities $\rho \in L^\infty_{\rm loc}(\mathbb R^d)$ such that $ \rho^{-1}\in L^\infty_{\rm loc}(\mathbb R^d)$, namely without assuming further conditions on $\rho(x) $ as $|x|\to\infty $, one can also prove existence and uniqueness of \emph{weak energy solutions} in the same spirit as \cite{GMPo} (see Proposition \ref{prop2c}). The point is that the uniqueness results of Theorem \ref{prthm6} hold for a more general notion of solution, and we shall use them as such.


\smallskip

The main goal of this paper is to study the large time behaviour of
solutions to problem \eqref{eq: regular}. To this end, similarly to the results recalled above in the local case, we shall distinguish two situations:
\begin{enumerate}[i)]
\item\label{slow} $\rho(x)\to 0$ \emph{slowly} as $|x|\to\infty$, in the sense that for a suitable $\gamma\in (0, 2s)$ there holds
\begin{equation}\label{decay2}
\lim_{|x|\to \infty}\rho(x)|x|^{\gamma}= c_\infty>0 \, ;
\end{equation}
\item\label{fast} $\rho(x)\to 0$ \emph{rapidly} as $|x|\to\infty$, in the sense that for a suitable $\gamma\in (2s, \infty)$ \eqref{decay} holds true.
\end{enumerate}
In case \ref{slow}) we shall describe the asymptotic behaviour of
solutions to problem \eqref{eq: measure}, namely with initial data
which can be positive finite measures. Such asymptotics is obtained in terms of a  Barenblatt-type solution to a proper
nonlocal singular problem, that is the unique solution $u^{c_\infty}_{M}$ to
\begin{equation}\label{e10}
\begin{cases}
c_\infty |x|^{-\gamma} u_t + (-\Delta)^s ( u^m ) = 0  &  \textrm{in} \ \mathbb{R}^d \times (0,\infty) \, , \\
c_\infty  |x|^{-\gamma} u  = M \delta &  \textrm{on } \mathbb{R}^d \times \{0\} \, ,
\end{cases}
\end{equation}
where $M>0$ is the (fixed) mass and $c_\infty$ is as in \eqref{decay2}. Existence and uniqueness of solutions to \eqref{e10} actually
follow from the results established in \cite{GMP2} for the more general problem \eqref{eq: measure}. In particular, they are ensured provided $ \gamma \in [0,2s) \cap [0,d-2s] $.

Coming back to the asymptotics of the solutions to the evolution equations considered, we shall show that
\begin{equation}\label{e104-intro}
\lim_{t\to \infty} \left\| u(t)- u^{c_\infty}_{M}(t) \right\|_{1,|x|^{-\gamma}} = \lim_{t \to \infty} \int_{\mathbb{R}^d} \left| t ^\alpha u(t^\kappa x , t)-u^{c_\infty}_M (x, 1) \right| |x|^{-\gamma} \mathrm{d}x = 0 \, ,
\end{equation}
where
\[
\alpha=(d-\gamma) \kappa \, , \quad \kappa=\frac 1{d(m-1)+2s-m\gamma} \, .
\]
In order to prove \eqref{e104-intro}, we partially follow the
general strategy used in the literature to prove similar
convergence results (see e.g.~\cite{FK, Vaz03, Vaz07, RV08, BBDG, VazConstr}). However, here several technical difficulties arise,
due to the simultaneous presence of the weight $\rho(x)$ and of
the nonlocal operator $(-\Delta)^s$. To overcome them, we adapt to
the present situation some ideas used in \cite{GMP2} to prove
existence. The point is that a different
argument in the final convergence step has to be used. Indeed, in
particular in \cite{RV08} and \cite{VazConstr}, convergence is proved in $L^\infty$ by exploiting regularity results for solutions
and for the Barenblatt profile, ensured by \cite{DiB} and \cite{AtC}, respectively. Since such regularity results by now are not available in our case, we cannot use the same techniques.
\smallskip

In case \ref{fast}), the long time behaviour of the minimal solution to problem \eqref{eq: regular} is deeply linked with the minimal solution $w$ to the following nonlocal sublinear elliptic equation:
\begin{equation}\label{19021}
(-\Delta)^s w = \rho\, w^\alpha \quad \textrm{in } \mathbb{R}^d \, ,
\end{equation}
where $\alpha = {1}/{m} \in (0,1)$. Note that the local case $s=1$ has been thoroughly studied (see e.g. \cite{BK,Rad} and references
therein). For general $s\in(0,1)$ it has been addressed in \cite{PT3}, following the same line of arguments of
\cite{BK}. However, in \cite{PT3} it is supposed that
\eqref{decay} holds true for $\gamma>d$ (with $d>4s$) and $\rho \ge 0$ (with $\rho\not
\equiv 0$). Furthermore, \emph{energy solutions} have been dealt
with. In the present work, existence of nontrivial \emph{very weak solutions} is
established whenever \eqref{decay} holds for $\gamma>2s$ (with $d>2s$). In doing
this, a central role will be played by the solution to the linear equation
\[
(-\Delta)^s V = \rho \quad\textrm{in }\mathbb{R}^d \, .
\]
We shall also establish uniqueness of very weak solutions to equation
\eqref{19021}, satisfying suitable extra conditions at infinity, assuming that $d>4s$ and that \eqref{decay} holds for
$\gamma>4s  $.

As for asymptotics, we shall prove that
\[
\lim_{t\to\infty}t^{\frac 1{m-1}}u(x,t)=(m-1)^{-\frac 1{m-1}}
w^{\frac 1m}(x) \, ,
\]
where $w$ is the minimal positive (very weak) solution to
\eqref{19021} with $\alpha=1/m$ and $ u $ is the minimal solution to \eqref{eq: regular}. Note that a similar result in
bounded domains, when $\rho\equiv 1$, has recently been shown in \cite{BSV} (see Remark \ref{ossBSV}).

\smallskip
Let us mention that by our methods
we cannot address the {\it critical} case in which $\rho(x)$ behaves like
$|x|^{-2s}$. In fact, in such case, we are not able either to construct
the asymptotic profile as in $(i)$ or the minimal solution to the sublinear elliptic equation as
in $(ii)$. Observe that for $s=1$ the long time behaviour of
solutions has been investigated in \cite{IS1} for $m=1$, and in \cite{NR,IS2} for $m>1$. For $0<s<1$ and $ \gamma=2s $, the asymptotic behaviour of solutions is then an interesting open problem.

\noindent{\bf Organization of the paper}. In Section \ref{mf} we
give the definitions of solution to problems \eqref{eq: regular}
and \eqref{eq: measure}; moreover, preliminary results concerning the well posedness of the problems are stated. As for long time
behaviour of solutions, our results both for fast decaying densities (Theorem
\ref{thmab2}) and for slowly decaying densities (Theorem \ref{thmab1})
are stated in Section \ref{mr}. In Section \ref{subl} we consider
the sublinear elliptic equation \eqref{19021}, and we show some new
existence and uniqueness results for the corresponding solutions in Theorems
\ref{prthm8} and \ref{02034}, which have also an independent interest. We
take advantage of such results in Section \ref{as2} in order to
prove Theorem \ref{thmab2}. Finally, in Section \ref{as1} we prove Theorem
\ref{thmab1}.

In Appendix \ref{rd}  the well posedness of problem
\eqref{eq: regular} for rapidly decaying densities is proved: here we improve in various directions previous results in \cite{PT1}.

\section{Preliminary results}\label{mf}
We start this section by providing a suitable definition of weak solution to problem \eqref{eq: regular}, which will be primarily interesting for the case of rapidly decaying densities. We shall always assume $\rho \in L_{\rm loc}^\infty(\mathbb{R}^d)$ and $\rho^{-1} \in L_{\rm loc}^\infty(\mathbb{R}^d)$. Hereafter, by the symbol $\dot{H}^{s}(\mathbb{R}^d)$ we shall denote the completion of $C_c^\infty(\mathbb{R}^d)$ w.r.t.\ the norm
$$  \left\| \phi \right\|_{ \dot{H}^{s} } = \left\| (-\Delta)^{\frac s 2} (\phi) \right\|_2  \ \ \ \forall \phi \in C_c^\infty(\mathbb{R}^d) \, . $$

\begin{den}\label{defsol1}
A nonnegative function $u$ is a weak solution to problem \eqref{eq: regular} corresponding to the nonnegative initial datum $u_0 \in L^1_\rho(\mathbb{R}^d)$ if:
\begin{itemize}
\item $u\in C([0,\infty); L^1_{\rho}(\mathbb{R}^d))\cap L^\infty(\mathbb{R}^d\times(\tau,\infty))$ for all $\tau>0$;
\item $u^m\in L^2_{\rm loc}((0, \infty); \dot{H}^{s}(\mathbb{R}^d))$;
\item for any $\varphi \in C^\infty_c(\mathbb{R}^d\times (0,\infty))$ there holds
\begin{equation}\label{e2}
\int_0^\infty \int_{\mathbb{R}^d}  u(x,t) \varphi_t(x,t) \, \rho(x) \mathrm{d}x \mathrm{d}t - \int_0^\infty \int_{\mathbb{R}^d} (-\Delta)^{\frac s 2}(u^m)(x,t) (-\Delta)^{\frac s 2} (\varphi)(x,t) \, \mathrm{d}x \mathrm{d}t = 0 \, ;
\end{equation}
\item $ \lim_{t \to 0} u(t) = u_0$ in $L^1_\rho(\mathbb{R}^d)$.
\end{itemize}
\end{den}

A classical notion in the literature is the following (see e.g.\ \cite[Section 8.1]{DQRV}).
\begin{den}\label{defsol2}
Let $u$ be a weak solution to problem \eqref{eq: regular} (according to Definition \ref{defsol1}). We say that $u$ is a \emph{strong solution} if, in addition, $u_t\in
L^\infty((\tau,\infty); L^1_\rho(\mathbb{R}^d))$ for every $\tau>0$.
\end{den}

Existence and uniqueness of weak solutions to problem \eqref{eq: regular}, by means of standard techniques (see e.g.\ \cite{DQRV2,DQRV,GMPo,PT1}), are discussed in Appendix \ref{rd}. The first result we provide reads as follows (for a sketch of proof see again Appendix \ref{rd} -- Parts I and II).
\begin{pro}\label{prop2c}
Let $\rho\in L^\infty_{\rm loc}({\mathbb R}^d)$ be positive and such that $\rho^{-1}\in L^\infty_{\rm loc}({\mathbb R}^d)$. Then there exists a \emph{unique} weak solution $u$ to problem \eqref{eq: regular}, in the sense of Definition \ref{defsol1}, which is also a strong solution in the sense of Definition \ref{defsol2}.
\end{pro}

Let us introduce the Riesz kernel of the $s$-Laplacian:
\begin{equation}\label{eq: def-riesz-pot}
I_{2s}(x) = \frac {k_{s,d}}{|x|^{d-2s}} \ \ \ \forall x \in \mathbb{R}^d \setminus \{0\}  \, ,
\end{equation}
where $k_{s,d}$ is a suitable positive constant that depends only on $s$ and $d$. Recall that for a sufficiently regular function $f$ there holds
$$ (-\Delta)^s \left(I_{2s} \ast f \right) = f \, , $$
namely the convolution against $I_{2s}$ represents the operator $(-\Delta)^{-s}$.

\subsection{Rapidly decaying densities}
Given a weak solution $u$ to \eqref{eq: regular} and any fixed $t_0 > 0$, let us set
\[
U(t_0;x,t)=\int_{t_0}^{t} u^m(x,\tau) \,
\mathrm{d}\tau \quad \forall(x,t) \in \mathbb{R}^d \times
(t_0,\infty) \, .
\]
When $\rho(x)$ is a density that decays \emph{sufficiently fast} as $|x|\to\infty$, we shall often need to deal with solutions to \eqref{eq: regular} which are meant in a more general sense with respect to the one of Definition \ref{defsol1}, namely what we call \emph{local strong solutions}. The corresponding definition is technical, and we leave it to Appendix \ref{rd} (see Definition \ref{defsol1u} below). The result we present here concerns existence and uniqueness of local strong solutions.
\begin{thm}\label{prthm6}
Let $\rho\in L^\infty({\mathbb R}^d)$ be positive and such that $\rho^{-1}\in L^\infty_{\rm loc}({\mathbb R}^d)$. Let $ u_0\in L^1_\rho(\mathbb R^d)$ be nonnegative. Assume in addition that $\rho(x)\leq C_0 \, |x|^{-\gamma}$ a.e.\ in $B_1^c$ for some $\gamma>2s$ and $C_0>0$. Then the weak solution to problem \eqref{eq: regular} provided by Proposition \ref{prop2c} is the \emph{minimal solution}
in the class of local strong solutions (according to Definition \ref{defsol1u} below) and satisfies
\begin{equation}\label{e10b}
U(t_0;x,t) \to 0 \quad \textrm{as} \ |x| \to \infty
\end{equation}
for any {\em fixed} $t_0>0$ and for all $t > t_0 $. More precisely, there holds
\begin{equation}\label{e10a-decay}
U(t_0;x,t) \leq C \, (I_{2s}\ast \rho)(x) \quad \textrm{for a.e. } (x,t) \in \mathbb{R}^d \times (t_0,\infty)
\end{equation}
for some $C>0$, whence \eqref{e10b} follows by Lemma \ref{l100} below. Furthermore:
\begin{enumerate}[(i)]
\item \label{prthm6-case-1} under the more restrictive assumption
that $d>4s$ and $\gamma\in (2s, d-2s] \cap(4s,\infty)$, the
solution is \emph{unique} in the class of local strong solutions
satisfying
\begin{equation}\label{e10c}
u^m \in L^1_{(1+|x|)^{-d+2s}}(\mathbb{R}^d\times (t_0,T)) \ \ \
\forall T>t_0>0 \, ;
\end{equation}
\item \label{prthm6-case-2} if $u_0$ is also bounded, then $u\in
L^\infty(\mathbb R^d\times (0,\infty))$ and all the above
results hold true with $t_0=0$ as well.
\end{enumerate}
\end{thm}

For the proof of Theorem \ref{prthm6}, we refer the reader to Appendix \ref{rd} -- Part III.
\begin{oss}
\rm Note that, as concerns uniqueness, for $d\geq 6s$ the
assumptions on $\gamma$ amount to $\gamma>2s$.
\end{oss}

\subsection{Slowly decaying densities}\label{bs}
In this subsection we deal with weights $\rho(x)$ which decay slowly as $ |x|\to\infty $. More precisely, we shall assume that the following hypotheses are satisfied:
\begin{equation}\label{eq: ass-rho-slow-bis}
c |x|^{- \gamma_0 } \le \rho(x) \le C |x|^{- \gamma_0 } \ \
\textrm{a.e.\ in } B_1 \quad \textrm{and} \quad  c |x|^{-\gamma}
\le \rho(x) \le C |x|^{-\gamma} \ \ \textrm{a.e.\ in } B_1^c
\end{equation}
for some $ \gamma \in [0,2s)$, $ \gamma_0 \in [0,\gamma] $ and $0< c
< C $. Note that $\rho(x)$ might possibly be unbounded as $x \to 0$.

Below we recall the definition of weak solution to the more general problem \eqref{eq: measure} given in \cite[Definition 3.1]{GMP2}. Before doing it, following the same notation as in \cite{Pierre}, we need to introduce some notions of convergence in measure spaces. Let $\mathcal{M}(\mathbb{R}^d)$ be the cone of positive, finite measures on $\mathbb{R}^d$. A sequence $\{ \mu_n \} \subset \mathcal{M}(\mathbb{R}^d) $ is said to converge to $ \mu \in \mathcal{M}(\mathbb{R}^d) $ in $ \sigma(\mathcal{M}(\mathbb{R}^d),$ $C_b(\mathbb{R}^d)) $ if
\[
\lim_{n \to \infty} \int_{\mathbb{R}^d} \phi(x) \, \mathrm{d}\mu_n = \int_{\mathbb{R}^d} \phi(x) \, \mathrm{d}\mu   \ \ \ \forall \phi \in C_b(\mathbb{R}^d) \, ,
\]
where $ C_b(\mathbb{R}^d) $ is the space of continuous, bounded functions in $ \mathbb{R}^d$. Analogous definitions hold for $ \sigma(\mathcal{M}(\mathbb{R}^d),C_c(\mathbb{R}^d))$ and $ \sigma(\mathcal{M}(\mathbb{R}^d),C_0(\mathbb{R}^d))$, where $ C_0(\mathbb{R}^d) $ is the closure of $ C_c(\mathbb{R}^d) $ w.r.t.\ the $ L^\infty(\mathbb{R}^d) $ norm.

\begin{den}\label{defsol2-barenblatt}
By a weak solution to problem \eqref{eq: measure}, corresponding to the initial datum $ \mu \in \mathcal{M}(\mathbb{R}^d)$, we mean a nonnegative function $ u $ such that:
\begin{equation}\label{e23}
u \in L^\infty( (0,\infty); L^1_{\rho}(\mathbb{R}^d)) \cap
L^\infty( \mathbb{R}^d \times (\tau , \infty ) ) \quad \forall \tau>0  \, ,
\end{equation}
\begin{equation}\label{e24}
u^m \in L^2_{\rm loc}((0,\infty);\dot{H}^s(\mathbb{R}^d)) \, ,
\end{equation}
\begin{gather}\label{e25}
\int_{0}^{\infty} \! \int_{\mathbb{R}^d}  u(x,t) \varphi_t(x,t) \, \rho(x) \mathrm{d}x \mathrm{d}t - \int_0^\infty \! \int_{\mathbb{R}^d} (-\Delta)^{\frac{s}{2}} (u^m)(x,t) \, (-\Delta)^{\frac{s}{2}}(\varphi) (x,t) \,  \mathrm{d}x \mathrm{d}t = 0 \\
\forall \varphi \in  C^\infty_c( \mathbb{R}^d \times (0,\infty)) \nonumber
\end{gather}
and
\[
\lim_{t \to 0} \rho \, u(t)  = \mu \ \ \ \textrm{in }  \sigma(\mathcal{M}(\mathbb{R}^d),C_b(\mathbb{R}^d)) \, .
\]
\end{den}

It is plain that, when $\mu = \rho \,
u_0 \in L^1(\mathbb{R}^d) $, a solution to \eqref{eq: regular}
with respect to Definition \ref{defsol1} is also a solution to
\eqref{eq: measure} with respect to Definition
\ref{defsol2-barenblatt}. However, Definition
\ref{defsol2-barenblatt} permits to handle more general
initial data (positive, finite measures). In particular, we cannot
ask $u \in C([0,\infty);L^1_\rho(\mathbb{R}^d))$.
Nevertheless, thanks to the fundamental Theorem
\ref{thm: teorema-esistenza} which we state below, when $\mu = \rho \, u_0 \in
L^1(\mathbb{R}^d) $ such two solutions do coincide (provided the
parameters $\gamma$, $s$ and $d $ meet the corresponding assumptions).

We recall now some well posedness results proved in
\cite{GMP2}. In fact, thanks to the theory developed therein, we can guarantee existence and uniqueness of weak solutions
to \eqref{eq: measure} (according to Definition \ref{defsol2-barenblatt}). Besides, Proposition 4.1 of \cite{GMP2} ensures that 
\begin{equation}\label{eq: cons-mass}
\int_{\mathbb{R}^d} u(x,t) \, \rho(x) \mathrm{d}x = \mu({\mathbb R}^d) \ \ \ \forall t>0 \, ,
\end{equation}
namely there is \emph{conservation of mass}. This is actually a
sole consequence of Definition \ref{defsol2-barenblatt} and the
hypothesis $\gamma\in[0,2s)$.

The next result is a crucial one but its proof follows along known lines.

\begin{thm}\label{thm: teorema-esistenza}
Let $d > 2s$. Assume that $\rho$ satisfies \eqref{eq: ass-rho-slow-bis} for some $ \gamma \in [0, 2s) \cap [0,d-2s]$ and $ \gamma_0 \in [0,\gamma]$. Then there exists a weak
solution $u$ to problem \eqref{eq: measure}, in the sense of
Definition \ref{defsol2-barenblatt}, which satisfies the smoothing
estimate
\begin{equation}\label{eq: smoothing-effect-general}
\left\| u(t) \right\|_\infty \le K \, t^{- \alpha } \mu(\mathbb{R}^d)^\beta   \ \ \ \forall t > 0 \, ,
\end{equation}
where $K$ is a suitable positive constant depending only on $ m $,
$\gamma$, $s$, $ d $, $C$  and
\begin{equation}\label{eq: smoothing-effect-exp-teorema}
\alpha= \frac{d-\gamma}{(m-1)(d-\gamma) + 2s-\gamma} \, , \  \ \
\beta= \frac{2s-\gamma}{(m-1)(d-\gamma) + 2s-\gamma }  \, .
\end{equation}
In particular, $u(t)\in L^1_{\rho}({\mathbb R}^d) \cap L^\infty({\mathbb R}^d)$ for
all $t>0$. Moreover, $u$ satisfies the energy estimates
\begin{equation}\label{eq: prima-esistenza-energy-1-teorema}
 \int_{t_1}^{t_2} \! \int_{\mathbb{R}^d} \left| (-\Delta)^{\frac{s}{2}} \left( u^m \right) (x,t) \right|^2  \mathrm{d}x \mathrm{d}t + \frac 1{m+1}\int_{\mathbb{R}^d} u^{m+1}(x,t_2) \, \rho(x)
 \mathrm{d}x  = \frac 1{m+1}\int_{\mathbb{R}^d} u^{m+1}(x,t_1) \, \rho(x) \mathrm{d}x
\end{equation}
and
\begin{equation}\label{eq: prima-esistenza-energy-2-teorema}
\int_{t_1}^{t_2} \! \int_{\mathbb{R}^d} \left| z_t(x,t)  \right|^2 \rho(x) \mathrm{d}x \mathrm{d}t  \le C^\prime
\end{equation}
for all $ t_2 > t_1 > 0$, where $ z=u^{\frac{m+1}{2}} $ and $C^\prime$ is a positive constant that
depends only on $m$, $ t_1 $, $ t_2 $ and on $ \int_{\mathbb{R}^d} u^{m+1}(x,t_1 /2) \, \rho(x) \mathrm{d}x $. Furthermore, such solution is \emph{unique}. 
\end{thm}

\begin{oss}
\rm 
\begin{enumerate}[(i)]
\item The smoothing effect \eqref{eq: smoothing-effect-general} can be proved as in \cite[Proposition 4.6]{GMP2}. In fact, such proof only relies on the validity of the fractional Sobolev inequality
\[
\left\| v \right\|_{2\frac{d-\gamma}{d-2s},\rho} \le \widetilde{C} \left\| (-\Delta)^s(v) \right\|_{2} \ \ \ \forall v \in \dot{H}^s(\mathbb{R}^d) \, ,
\]
which, thanks to the assumptions on $\rho$, is a trivial consequence of
\begin{equation}\label{eq: sobo-2}
\left\| v \right\|_{2\frac{d-\gamma}{d-2s},-\gamma} \le C_{S,\gamma} \left\| (-\Delta)^s(v) \right\|_{2} \ \ \ \forall v \in \dot{H}^s(\mathbb{R}^d) \, .
\end{equation}
For the validity of \eqref{eq: sobo-2}, we refer the reader to \cite[Lemma 4.5]{GMP2} and references quoted.

\item Thanks to the results of \cite[Section 3.1]{GMP2} (which in turn go back to \cite[Section 8.1]{DQRV}), or to the discussion in Appendix \ref{rd} -- Part I (which applies to slowly decaying densities as well), we have that the solutions provided by Theorem \ref{thm: teorema-esistenza} are also strong. In particular, they belong to $C((0,\infty);L^1_{\rho}(\mathbb{R}^d))$.

\item For $d \ge 4s $ the assumptions of Theorem \ref{thm: teorema-esistenza} on $\gamma$ amount to $\gamma \in [0,2s)$.
\end{enumerate}
\end{oss}

\section{Main results: large time behaviour of solutions}\label{mr}
In this section we state our main results for the asymptotics (as $t\to\infty$) of the solutions to problems \eqref{eq: regular} and \eqref{eq: measure} provided by Proposition \ref{prop2c} and Theorem \ref{thm: teorema-esistenza}, respectively.

\subsection{Rapidly decaying densities}
As concerns solutions to \eqref{eq: regular} when $\rho(x)$ is a density that decays sufficiently fast as $|x|\to\infty$, we have the following result.
\begin{thm}\label{thmab2}
Let $\rho\in C^{\sigma}_{\rm loc}(\mathbb R^d)$ for some
$\sigma\in(0,1)$, with $\rho>0$. Let $u_0\in L^1_{\rho}(\mathbb{R}^d)$
be nonnegative. Assume in addition that $\rho(x)\leq C_0
|x|^{-\gamma}$ in $B_1^c$ for some $\gamma>2s$ and $C_0>0$. Let
$u$ be the (minimal) weak solution to problem \eqref{eq: regular}
provided by Proposition \ref{prop2c} and $ w $ be the very weak
solution to the sublinear elliptic equation \eqref{19021}, with
$\alpha=1/ m$, provided by Theorem \ref{prthm8} below (which is
also minimal in the class of solutions specified by the
corresponding statement). Then, 
\begin{equation}\label{eq101}
\lim_{t\to\infty}t^{\frac 1{m-1}} \, u(x,t)=(m-1)^{-\frac 1{m-1}}
\, w^{\frac 1m}(x)
\end{equation}
monotonically and in $L^p_{\rm loc}(\mathbb R^d)$ for all $p\in [1,\infty)$.
\end{thm}

\begin{oss}\label{ossBSV} \rm
Note that, if we introduce the {\it relative error} $\left|{u}/{\mathcal{U}}-1\right|$, where 
$$
\mathcal{U}(x,t)=(m-1)^{-\frac{1}{m-1}} \, t^{-\frac{1}{m-1}}w^{\frac{1}{m}}(x) $$
(we shall prove below that $w$ is strictly positive), then Theorem \ref{thmab2} implies that
\[
\left| \frac{u(x,t)}{\mathcal{U}(x,t)}-1 \right| \to 0\quad \textrm{as } t \to \infty \, ,
\]
at least locally. When $\rho \equiv 1$ and the problem is posed in a bounded domain, in \cite{BSV} the authors investigate the rate of convergence, uniformly in space, of such relative error.
\end{oss}

\begin{oss}
\rm Under the same assumptions as in Theorem \ref{thmab2}, with in
addition $d>4s$ and $\gamma>4s$, thanks to the uniqueness results
of Theorem \ref{prthm6} and Theorem \ref{prthm8}, we can
read the above asymptotic result as follows: any nontrivial local
strong solution $u$ to \eqref{eq: regular} satisfying $u^m\in
L^1_{(1+|x|)^{-d+2s}}(\mathbb R^d\times(t_0, T))$ for all
$T>t_0>0$ converges, in the sense of \eqref{eq101}, to the unique
nontrivial local weak and very weak solution $w$ to \eqref{19021}
(with $\alpha=1/m$) satisfying $w\in L^1_{(1+|x|)^{-d+2s}}(\mathbb
R^d)$.
\end{oss}

\subsection{Slowly decaying densities}\label{sect: sect-bar}
In the analysis of the long time behaviour of solutions to \eqref{eq: measure} when $\rho(x)$ is
density that decays \emph{slowly} as $|x|\to\infty$, a major role is played by the solution to the same
problem in the particular case $\rho(x)=c_\infty |x|^{-\gamma}$ and $\mu= M \delta$,
 for given positive constants $c_\infty$ and $M$ (namely, the solution to \eqref{e10}). From now on we shall denote such solution as $u^{c_\infty}_M$.

Let us define the positive parameters $\alpha$ and $\kappa$ as follows:
\begin{equation}\label{alpha-kappa}
\alpha=(d-\gamma) \kappa \, , \quad
\kappa=\frac1{(m-1)(d-\gamma)+2s-\gamma} \, .
\end{equation}
Notice that $\alpha$ is the same parameter appearing in \eqref{eq: smoothing-effect-exp-teorema}. It is immediate to check that, for any given $\lambda>0$, the function
\[
u^{c_\infty}_{M,\lambda}(x,t)=\lambda^\alpha u^{c_\infty}_M(\lambda^\kappa x,\lambda t)
\]
is still a solution to problem \eqref{e10}. Hence, as a consequence of the uniqueness result contained in Theorem \ref{thm: teorema-esistenza}, $u^{c_\infty}_{M,\lambda}$ and $ u^{c_\infty}_{M} $ must necessarily coincide, that is
\begin{equation}\label{eq: barenblatt-rescaled-identity}
u^{c_\infty}_{M}(x,t)=\lambda^\alpha u^{c_\infty}_M(\lambda^\kappa x,\lambda t) \ \ \ \forall t,\lambda>0 \, , \ \textrm{for a.e.\ } x \in \mathbb{R}^d \, .
\end{equation}
As already mentioned, the special solution $u^{c_\infty}_M$, thanks to the self-similarity identity \eqref{eq: barenblatt-rescaled-identity} it satisfies, will be crucial in the study of the asymptotic behaviour of \emph{any} solution to \eqref{eq: measure} (provided $\rho$ complies with \eqref{eq: ass-rho-slow-lim} as well). This is thoroughly analysed in Section \ref{as1}.

\smallskip

Our main result concerning the asymptotics of solutions to \eqref{eq: measure} is the following.
\begin{thm}\label{thmab1}
Let $d > 2s$. Assume that $\rho$ satisfies \eqref{eq: ass-rho-slow-bis} for some $ \gamma \in [0, 2s) \cap [0,d-2s]$ and $ \gamma_0 \in [0,\gamma]$, with in addition
\begin{equation}\label{eq: ass-rho-slow-lim}
\lim_{|x|\to\infty}\rho(x)|x|^\gamma=c_\infty>0 \, .
\end{equation}
Let $u$ be the unique weak solution to problem \eqref{eq: measure}, in the sense of Definition \ref{defsol2-barenblatt}, provided by Theorem \ref{thm: teorema-esistenza} and corresponding to $\mu \in \mathcal{M}(\mathbb{R}^d)$ as initial datum, with $\mu({\mathbb R}^d)=M>0$. Then,
\begin{equation}\label{e104}
\lim_{t \to \infty} \left\|u(t)-u^{c_\infty}_M(t)
\right\|_{1,|x|^{-\gamma}}=0
\end{equation}
or equivalently
\begin{equation}\label{e104-equiv}
\lim_{t \to \infty} \int_{\mathbb{R}^d} \left| t^\alpha u(t^\kappa
x,t) - u^{c_\infty}_M(x,1) \right| |x|^{-\gamma} \,\mathrm{d}x =0
\, ,
\end{equation}
where $u^{c_\infty}_M$ is the Barenblatt solution defined as the unique solution to problem \eqref{e10}, and the parameters $\alpha$, $\kappa$ are as in \eqref{alpha-kappa}.
\end{thm}

Notice once again that the range of $\gamma$ for which the above theorem holds true simplifies to $[0,2s)$ when $d \ge 4s$, which is, to some extent, the maximal one for which one can expect a similar result.

Theorem \ref{thmab2} will be proved in Section \ref{as2}, while Theorem \ref{thmab1} will be proved in Section \ref{as1}.

\section{A fractional sublinear elliptic equation}\label{subl}
Prior to analysing the asymptotic behaviour of solutions to \eqref{eq: regular} when $\rho(x)$ is a density that decays fast as $|x|\to \infty$ (discussed in Section \ref{as2}), we need to study the sublinear elliptic equation \eqref{19021}, which naturally arises from such asymptotic analysis.

Let us recall that if $\varphi$ is a smooth and compactly supported function defined in $\mathbb R^d$, we can consider its $s$-harmonic extension $\D{E}(\varphi)$ to the upper half-space $\mathbb{R}^{d+1}_+ = \{(x,y): \, x \in \mathbb R^d, \, y>0 \}$, namely the unique smooth and bounded solution to the problem
\[
\begin{cases}
\D{div} \left( y^{1-2s}\nabla{\D{E}(\varphi)} \right)=0 & \textrm{in } \mathbb{R}^{d+1}_+ \, , \\
\D{E}(\varphi)=\varphi & \textrm{on } \partial \mathbb{R}^{d+1}_+ = \mathbb{R}^d \times \{y= 0 \} \, .
\end{cases}
\]
It has been proved (see e.g.\ \cite{CafS,DQRV,CS1}) that
\[
-\mu_s \lim_{y\to 0^+} y^{1-2s}\frac{\partial \D{E}(\varphi)}{\partial y}(x,y)=\left(-\Delta \right)^{s}(\varphi)(x) \quad \forall x \in \mathbb{R}^d \, ,
\]
where $\mu_s=\frac{2^{2s-1}\Gamma(s)}{\Gamma(1-s)}$. It is therefore convenient to define the operators
\begin{gather*}
L_s = \D{div} \left( y^{1-2s}\nabla \right) , \\
\frac{\partial }{\partial y^{2s}} = - \mu_s \lim_{y\to 0^+} y^{1-2s}\frac{\partial }{\partial y} \, .
\end{gather*}
We also denote by $X^{s}$ the completion of $C^\infty_c(\mathbb{R}^{d+1}_+ \cup \partial \mathbb{R}^{d+1}_+)$ w.r.t.\ the norm
\[
\| \psi \|_{X^s} = \left(\mu_s \int_{\mathbb{R}^{d+1}_+} y^{1-2s}\left|\nabla \psi(x,y) \right|^2  \dd x \dd y\right)^\frac{1}{2} \quad \forall  \psi \in C^\infty_c(\mathbb{R}^{d+1}_+ \cup \partial \mathbb{R}^{d+1}_+) \, .
\]
Furthermore, by the symbol $X^{s}_{\rm loc}$, we shall mean the space of all functions $v$ such that $\psi v \in X^{s}$ for any $\psi \in C^\infty_c(\mathbb{R}^{d+1}_+ \cup \partial \mathbb{R}^{d+1}_+)$.

It is possible to prove that there exists a well defined notion of trace on $\partial \mathbb{R}^{d+1}_+$ for every function in $X^{s}$ (see e.g.\ \cite[Section 3.2]{DQRV}, \cite[Section 2]{BCdP} or \cite[Section 3.1]{CS1}). Moreover, for every $v \in \dot{H}^s(\mathbb{R}^d)$ there exists a unique extension $\D{E}(v)\in X^{s}$ such that
$$ \D{E}(v)(x,0)=v(x) \quad  \textrm{for a.e.\ } x \in \mathbb{R}^d  $$
and
$$ {\mu_s} \int_{\mathbb{R}^{d+1}_+}  y^{1-2s} \langle \nabla{\D{E}(v)}, \nabla{\psi} \rangle(x,y) \, \dd x \dd y = \int_{\mathbb{R}^d} (-\Delta)^s(v)(x) \,  (-\Delta)^s(\psi)(x,0) \, \dd x  $$
for any $\psi \in C^\infty_c(\mathbb{R}^{d+1}_+ \cup \partial \mathbb{R}^{d+1}_+)$. 

Having at our disposal the above tools, we can provide suitable weak formulations of problem \eqref{19021} which deal with the harmonic extension. In fact, at a formal level, looking for a solution $w$ to \eqref{19021} is the same as looking for a pair of functions $(w,\tilde{w})$ solving the problem
\begin{equation}\label{e20d}
\begin{cases}
L_{s} \tilde w = 0 & \textrm{in } \mathbb{R}^{d+1}_+ \, , \\
\tilde{w}=w & \textrm{on } \partial \mathbb{R}^{d+1}_+ \, , \\
\displaystyle\frac{\partial \tilde{w}}{\partial y^{2s}}= \rho \, w^\alpha  & \textrm{on }  \partial \mathbb{R}^{d+1}_+ \, , \\
\end{cases}
\end{equation}
with $0<\alpha<1$.

\begin{den}\label{defsol01}
A local weak solution to problem \eqref{e20d} is a bounded nonnegative function $w$ such that, for some nonnegative $\tilde w \in X^{s}_{\rm loc} \cap L^\infty_{\rm loc}(\mathbb{R}^{d+1}_+ \cup \partial\mathbb{R}^{d+1}_+) $ (what we call a \emph{local extension} for $w$), there holds $\tilde{w}|_{\partial \mathbb{R}^{d+1}_+}=w$ and
\[
\int_{\mathbb{R}^d}  w^\alpha(x)\,  \psi(x,0) \, \rho(x) \dd x =  {\mu_s}\int_{\mathbb{R}^{d+1}_+}  y^{1-2s} \langle \nabla{\tilde{w}} , \nabla{\psi} \rangle(x,y) \, \dd x \dd y
\]
for any $\psi \in C^\infty_c(\mathbb{R}^{d+1}_+ \cup \partial\mathbb{R}^{d+1}_+)$.
\end{den}

\begin{den}\label{defsol01-very-weak}
A bounded, nonnegative function $w$ is a very weak solution to problem \eqref{19021} if it satisfies
\[
\int_{\mathbb{R}^d}
w^\alpha(x) \varphi(x) \, \rho(x)  {\rm d}x = \int_{\mathbb{R}^d} w(x) (-\Delta)^{s}(\varphi)(x) \, {\rm d}x
\]
for any  $\varphi \in C^\infty_c(\mathbb{R}^d)$.
\end{den}

\begin{den}\label{defsol01-weak-energy}
A nonnegative function $w \in \dot{H}^s(\mathbb{R}^d)$ is a weak solution to problem \eqref{19021} if it satisfies
\begin{equation} \label{weak-ener}
\begin{aligned}
 \int_{\mathbb{R}^d} w^\alpha(x) \psi(x,0) \, \rho(x)  {\rm d}x = & \int_{\mathbb{R}^d}  (-\Delta)^{\frac{s}{2}}(w)(x) (-\Delta)^{\frac{s}{2}}(\psi)(x,0) \, {\rm d}x \\
= & {\mu_s}\int_{\mathbb{R}^{d+1}_+}  y^{1-2s} \langle \nabla{\D{E}(w)} , \nabla{\psi} \rangle(x,y) \, \dd x \dd y
\end{aligned}
\end{equation}
for any $\psi \in C^\infty_c(\mathbb{R}^{d+1}_+ \cup \partial \mathbb{R}^{d+1}_+)$.
\end{den}
Note that a bounded weak solution is a solution to \eqref{19021} in the sense of both Definition \ref{defsol01} and Definition \ref{defsol01-very-weak}.

What follows in this section aims at studying existence and uniqueness of solutions to \eqref{e20d} (and \eqref{19021}), according to Definition \ref{defsol01} (and \ref{defsol01-very-weak}, \ref{defsol01-weak-energy}). Our results are the following.

\begin{thm}[existence]\label{prthm8}
Let $\alpha \in (0,1)$. Let $\rho \in C^{\sigma}_{\rm loc}(\mathbb R^d)$ (for some $\sigma\in (0,1)$) be strictly positive and
such that $\rho(x)\leq C_0 |x|^{-\gamma}$ in $B_1^c$ for some $\gamma>2s$ and $C_0>0$. Then there exists a local weak solution $w$ to problem \eqref{e20d}, which is minimal in the class of nonidentically zero local weak solutions (according to Definition \ref{defsol01}). Moreover, $w$ is a very weak solution to \eqref{19021} (in the sense of Definition \ref{defsol01-very-weak}) and satisfies the estimate
\begin{equation}\label{e59}
w(x) \leq C (I_{2s} \ast \rho)(x) \quad  \forall x \in \mathbb{R}^d
\end{equation}
for some $C>0$.

Finally, if $\gamma $ complies with the more restrictive condition
\begin{equation}\label{eq:
cond-ener} \gamma > 2s + \frac{d-2s}{\alpha+2} \, ,
\end{equation}
then $w$ is also a weak solution to \eqref{19021} (according to Definition \ref{defsol01-weak-energy}).
\end{thm}

\begin{thm}[uniqueness]\label{02034}
Let $d>4s$ and $\alpha \in (0,1)$. Let $\rho\in C^{\sigma}_{\rm
loc}(\mathbb R^d)$ (for some $\sigma \in (0,1)$) be strictly
positive and such that $\rho(x)\leq C_0 |x|^{-\gamma}$ in $B_1^c$
for some $\gamma > 4s $ and $C_0>0$. Let $\underline{w}$
be the minimal solution to problem \eqref{e20d} provided by
Theorem \ref{prthm8}. Let $w$ be any other local weak solution to
problem \eqref{e20d} (according to Definition \ref{defsol01}),
which is also a very weak solution to problem \eqref{19021}
(according to Definition \ref{defsol01-very-weak}) and such that $
w \not \equiv 0$ and $ w \in L^1_{(1+|x|)^{-d+2s}}(\mathbb{R}^d)$.
Then $w=\underline{w}$.
\end{thm}

The next lemma, which provides us with elementary estimates from above for the Riesz potential of $ \rho $, is key to our analysis. We skip the proof since it is just a matter of routine computations which exploit decay and integrability properties of $ I_{2s} $ and $\rho$.
\begin{lem}\label{l100}
Let $d>2s$ and $ \rho \ge 0  $ be a measurable function. Assume in
addition that $ \rho(x) \leq C (1+|x|)^{-\gamma} $ for some
$\gamma>2s$ and $C>0$. Then, $ I_{2s} \ast \rho $ is a nonnegative
continuous function and there exists a constant $K>0$ such that
\begin{equation*}\label{lem: riesz-rho-decay}
(I_{2s} \ast \rho)(x) \le K \left(1 + |x| \right)^{-\kappa} \quad
\forall x \in \mathbb R^d \, ,
\end{equation*}
where:
\begin{enumerate}[(a)]
\item if $ \gamma < d $, $ \kappa=\gamma-2s $; \item if $ \gamma=d
$, $ \kappa = d-2s-\varepsilon $ for all $ \varepsilon>0 $ (with
$K=K(\varepsilon)$); \item if $ \gamma>d $, $ \kappa=d-2s  $.
\end{enumerate}
\end{lem}

In view of Theorem \ref{prthm8} and Lemma \ref{l100}, it is apparent that under the assumptions of Theorem \ref{02034} the minimal solution $ \underline{w} $ does belong to $L^1_{(1+|x|)^{-d+2s}}(\mathbb{R}^d)$.

\subsection{Existence}\label{sottosez-not}
Here we shall prove all the properties of $w$ claimed in Theorem \ref{prthm8}, \emph{except} the fact that $w$ is a very weak solution to problem \eqref{19021} in the sense of Definition \ref{defsol01-very-weak} for all $\gamma>2s$. This will be in fact a consequence of the asymptotic analysis of Section \ref{as2}.

Let us start off with some preliminaries. We consider first the following problem: find $(w_R,\tilde{w}_R)$ such that
\begin{equation}\label{e20c}
\begin{cases}
L_{s} \tilde{w}_R = 0  &  \textrm{in } \Omega_R \, , \\
\tilde{w}_R = 0 & \textrm{on } \Sigma_R \, , \\
\tilde{w}_R = w_R & \textrm{on } \Gamma_R  \, , \\
\displaystyle \frac{\partial \tilde{w}_R}{\partial y^{2s}}= \rho \, w_R^\alpha & \textrm{on } \Gamma_R  \, , \\
\end{cases}
\end{equation}
where $\Omega_R= \{(x,y) \in \mathbb{R}^{d+1}_+ : \ |(x,y)|<R\}$, $ \Sigma_R=\partial \Omega_R \cap \{y>0\}$ and $ \Gamma_R=\partial \Omega_R \cap \{y=0\}$. We denote by $X_0^{s}(\Omega_R)$ the completion of $C^\infty_c(\Omega_R \cup \Gamma_R)$ w.r.t.\ the norm
\[
\| \psi \|_{X_0^{s}(\Omega_R)} = \left(\mu_s \int_{\Omega_R} y^{1-2s}|\nabla \psi(x,y) |^2 \, \dd x \dd y\right)^\frac{1}{2} \quad \forall \psi \in C^\infty_c(\Omega_R \cup \Gamma_R) \, .
\]

\begin{den}\label{defsolc}
A weak solution to problem \eqref{e20c} is a pair of nonnegative functions $(w_R,\tilde{w}_R)$ such that:
\begin{itemize}
\item{} $w_R^\alpha \in L^1(B_R)$, $ \tilde{w}_R \in X_0^{s}(\Omega_R)$;
\item{} $ \tilde{w}_R|_{\Gamma_R}= w_R$;
\item{} for any $ \psi \in C^\infty_c(\Omega_R \cup \Gamma_R)$ there holds
\begin{equation}\label{03011-4}
\int_{B_R} w_R^\alpha(x) \, \psi(x,0)\, \rho(x) \mathrm{d}x = {\mu_s}\int_{\Omega_R}  y^{1-2s} \langle \nabla{\tilde{w}_R} , \nabla{\psi} \rangle(x,y) \, \dd x \dd y   \, .
\end{equation}
\end{itemize}
\end{den}
The next existence result concerning problem \eqref{e20c} can be proved by standard variational methods (see e.g.\ \cite{BCdP}).

\begin{pro}\label{190213}
Let $\alpha\in(0,1)$. Let $\rho\in L^\infty_{\rm loc}({\mathbb R}^d)$ be positive and such that $\rho^{-1}\in L^\infty_{\rm loc}({\mathbb R}^d)$. Then there exists a non-identically zero weak solution $(w_R,\tilde{w}_R)$ to problem \eqref{e20c}, in the sense of Definition \ref{defsolc}.
\end{pro}

The following regularity and comparison results for problem \eqref{e20c} will be crucial in the proof of Theorem \ref{prthm8} (specially as for minimality).

\begin{pro}\label{prop-CS1}
Let $\alpha\in(0,1)$. Let $\rho\in L^\infty_{\rm loc}({\mathbb
R}^d)$ be nonnegative and such that $\rho^{-1} \in L^\infty_{\rm
loc}({\mathbb R}^d)$.
\begin{enumerate}[(i)]
\item\label{st-1} Take a subsolution $(w_{R}^{(1)},
\tilde{w}^{(1)}_{R})$ and a supersolution $(w^{(2)}_{R},
\tilde{w}^{(2)}_{R})$ to problem \eqref{e20c} (in a weak sense, in
agreement with Definition \ref{defsolc}). Assume that
$\tilde{w}^{(1)}_{R},\tilde{w}^{(2)}_{R} \geq 0$ a.e.\ in
$\Omega_R$, $w_{R}^{(1)} \ge 0$ a.e.\ in $B_R$, $w^{(2)}_{R}>0$
a.e.\ in $B_R$ and $\tilde{w}^{(1)}_{R}|_{\Sigma_R} \leq
\tilde{w}^{(2)}_{R}|_{\Sigma_R}$ a.e.\ in $\Sigma_R$. Then
$\tilde{w}^{(1)}_{R} \leq \tilde{w}^{(2)}_{R}$ a.e.\ in $\Omega_R$
and $w_{R}^{(1)} \leq w^{(2)}_{R}$ a.e.\ in $B_R$.
\item\label{st-2} Suppose in addition that $\rho\in
C^{\sigma}_{\rm loc}(\mathbb R^d)$ for some $\sigma \in (0,1)$.
Let $(w_R,\tilde{w}_R)$ be a weak solution to problem
\eqref{e20c}, in the sense of Definition \ref{defsolc}, such that
$ w_R \in L^\infty(B_R)$ and $\tilde{w}_R \in L^\infty(\Omega_R)$.
Then (in particular) $\tilde{w}_R \in C(\overline{\Omega}_r)$ for
all $ 0<r<R $ and either $(w_R,\tilde{w}_R) \equiv (0,0)$ or
$w_R>0 $ in $B_R$ and $\tilde{w}_R>0$ in $\Omega_R $.
\end{enumerate}
\end{pro}
\begin{proof} Statement \eqref{st-1} follows by performing minor modifications to the proof of \cite[Lemma 5.3]{BCdP}. Actually the strategy of proof goes back to the pioneering paper \cite{BK}: let us mention that the strict positivity of the supersolution and the fact that the nonlinearity is sublinear are essential.
Statement \eqref{st-2} is due to the regularity results in
\cite{CS1}. In fact, since $(w_R,\tilde{w}_R)$ is bounded and $
\rho(x) $ and $ f(w):=w^\alpha $ are H\"{o}lder functions, Lemma
4.5 of \cite{CS1} ensures that $\tilde w_R$ and $ \frac{\partial
\tilde{w}_R }{\partial y^{2s}} $ are also H\"{o}lder continuous in
$\overline{\Omega}_r$ for all $ 0<r<R $. Corollary 4.12 of
\cite{CS1} then entails the assertion (the same argument works
upon replacing $ d(x)u(x) $ there with $ -\rho(x)u^\alpha(x) $).
\end{proof}

We are now in position to prove Theorem \ref{prthm8} as concerns the existence of a minimal local weak solution to \eqref{e20d}. The fact that such solution is also a very weak solution to \eqref{19021} (according to Definition \ref{defsol01-very-weak}) will be deduced in the end of the proof of Theorem \ref{thmab2} in Section \ref{as2}.

\begin{proof}[Proof of Theorem \ref{prthm8} (first part)]
For any $R>0$, by Proposition \ref{190213} we know that there exists a nontrivial solution $(w_R, \tilde w_R)$ to problem \eqref{e20c}.
Let now $(\chi_R, \tilde{\chi}_R)$ be the unique regular solution to the problem
\[
\begin{cases}
L_s \tilde \chi_R =0 & \textrm{in } \mathcal C_R \, , \\
\tilde \chi_R=0 & \textrm{on } \partial\mathcal C_R\cap \{y>0\} \, , \\
\tilde \chi_R=\chi_R & \textrm{on } \Gamma_R \, , \\
\displaystyle\frac{\partial \tilde\chi_R}{\partial y^{2s}} = \rho  & \textrm{on } \Gamma_R  \, ,
\end{cases}
\]
where $\mathcal C_R = B_R\times \{y>0 \}$. By standard results (see e.g.\ \cite{CKL}), we have:
\begin{equation}\label{e87}
\tilde \chi_R(x,y)= \int_{B_R} G_R((x,y),z) \, \rho(z) \mathrm{d}z \quad \forall (x,y) \in \mathcal C_R  \, ,
\end{equation}
where $G_R((x, y),z)$ (let $(x,y)\in \mathcal C_R$ and $ z\in B_R)$ is the Green function, namely the solution of
\[
\begin{cases}
L_s  G_R(\cdot, z) =0& \textrm{in } \mathcal C_R \, , \\
G_R(\cdot, z)=0 & \textrm{on } \partial\mathcal C_R\cap \{y>0\} \, , \\
\displaystyle\frac{\partial G_R(\cdot, z)}{\partial y^{2s}} = \delta_z & \textrm{on } \Gamma_R \, ,
\end{cases}
\]
for each $z\in B_R.$ It is well known that the Green functions are positive and ordered w.r.t.\ $R$, that is, if $R_1 \le R_2$ then
\begin{equation}\label{eq: func-red-order}
0 < G_{R_1} \leq G_{R_2} \quad \textrm{in } \overline{\mathcal{C}}_{R_1} \, .
\end{equation}
Furthermore, they are all bounded from above by the Green function $G_{+}$ for the half-space:
\begin{equation}\label{e88}
G_R((x,y),z) \leq G_{+}((x,y),z)\quad \forall (x,y) \in \overline{\mathcal{C}}_R \, , \ \forall z \in B_R \, , \ \forall R>0 \, ,
\end{equation}
where
\[
G_+((x,y),z) = \frac{k_{s,d}}{|((x-z), y)|^{d-2s}} \quad  \forall (x,y) \in \mathbb{R}^{d+1}_+ \, , \ \forall z \in \mathbb{R}^d
\]
(for the same constant $k_{s,d}$ appearing in \eqref{eq: def-riesz-pot}). The function $G_+$ solves
\[
\begin{cases}
L_s  G_+(\cdot, z) =0 & \textrm{in }  \mathbb{R}^{d+1}_+ \, , \\
\displaystyle \frac{\partial G_+(\cdot, z)}{\partial y^{2s}} = \delta_z & \textrm{on } \partial \mathbb{R}^{d+1}_+ \, ,
\end{cases}
\]
for each $z\in \mathbb{R}^d$ (see again \cite{CKL} and also \cite{DBVol}). From \eqref{eq: def-riesz-pot}, \eqref{e87} and \eqref{e88} it clearly follows that, for any $R>0$ and any $(x,y) \in \overline{\mathcal C}_R $,
\begin{equation}\label{e89}
\tilde{\chi}_R(x,y) \leq \int_{\mathbb R^d} G_+((x,y),z) \, \rho(z) {\rm d}z \le \int_{\mathbb R^d} G_+((x,0),z) \, \rho(z) {\rm d}z  = (I_{2s} \ast \rho)(x) \leq \left\| I_{2s}\ast \rho \right\|_\infty= \widehat{C}
\end{equation}
(for the last inequality, see Lemma \ref{l100}). Now note that, for any test function $\psi$ as in Definition \ref{defsolc}, we have:
\begin{equation}\label{e889}
{\mu_s}\int_{\Omega_R}  y^{1-2s} \langle \nabla{\tilde{\chi}_R} , \nabla{\psi} \rangle(x,y) \, \dd x \dd y = {\mu_s}\int_{\mathcal{C}_R}  y^{1-2s} \langle \nabla{\tilde{\chi}_R} , \nabla{\psi} \rangle(x,y) \, \dd x \dd y = \int_{B_R} \psi(x,0) \, \rho(x) {\rm d}x \, .
\end{equation}
If we choose any $\overline{C} \geq \widehat{C}^\frac{\alpha}{1-\alpha}$, then the function $(\overline{C}\chi_R , \overline{C}\tilde{\chi}_R)$ is a supersolution to problem \eqref{e20c}. In fact, thanks to \eqref{e89} and \eqref{e889}, in this case there holds
\[
{\mu_s}\int_{\Omega_R}  y^{1-2s} \left\langle \nabla{\left(\overline{C}\tilde{\chi}_R\right)} , \nabla{\psi} \right\rangle(x,y) \, \dd x \dd y  = \int_{B_R} \overline{C} \psi(x,0) \, \rho(x) {\rm d}x \geq \int_{B_R} \left[\overline{C}\chi_R(x)\right]^\alpha \psi(x,0) \, \rho(x) {\rm d}x
\]
for all \emph{nonnegative} $\psi$ as above. Hence, thanks to \eqref{e87} and \eqref{eq: func-red-order}, we are in position to apply the comparison principle provided by Proposition \ref{prop-CS1}-\eqref{st-1} with the choices $(w_{R}^{(1)}, \tilde{w}^{(1)}_{R})=(w_R,\tilde{w}_R)$ and $(w^{(2)}_{R}, \tilde{w}^{(2)}_{R})=(\overline{C}\chi_R , \overline{C}\tilde{\chi}_R)$, to get:
\begin{equation}\label{190219}
\tilde{w}_R \leq \overline{C} \tilde{\chi}_R \quad \textrm{a.e. in } \Omega_R \, ,
\end{equation}
and
\begin{equation}\label{190219-bis}
w_R \leq \overline{C} \chi_R \quad \textrm{a.e. in } B_R \, .
\end{equation}
In particular, by \eqref{e89}, \eqref{190219} and \eqref{190219-bis} we deduce that $w_R \in L^\infty(B_R)$ and $\tilde{w}_R \in L^\infty(\Omega_R)$. We can now exploit Proposition \ref{prop-CS1}-\eqref{st-2} and infer that
\begin{equation}\label{17042}
\tilde{w}_R > 0 \quad \textrm{in } \Omega_R
\end{equation}
and
\begin{equation}\label{17042-bis}
{w}_R > 0 \quad \textrm{in } B_R \, .
\end{equation}
Let $0<R_1<R_2$. The strict positivity, for all $R>0$, of $(w_R,\tilde{w}_R)$ given by \eqref{17042} and \eqref{17042-bis} allows us to apply again Proposition \ref{prop-CS1}-\eqref{st-1}, this time with the choices $(w_{R}^{(1)}, \tilde{w}^{(1)}_{R})=(w_{R_1},\tilde{w}_{R_1})$ and $(w^{(2)}_{R}, \tilde{w}^{(2)}_{R})=(w_{R_2}, \tilde{w}_{R_2})$, to get:
\begin{equation}\label{e90}
\tilde{w}_{R_1} \leq \tilde{w}_{R_2} \quad \textrm{in } \Omega_R  \, , \quad {w}_{R_1} \leq {w}_{R_2} \quad \textrm{in } B_R \quad \forall R_2>R_1>0 \, .
\end{equation}
We need to pass to the limit on $(w_R,\tilde{w}_R)$ as $R\to\infty$. Given any \emph{fixed} $\eta\in C^\infty_c(\mathbb{R}^{d+1}_+ \cup \partial \mathbb{R}^{d+1}_+)$, for every $R>0$ large enough we can pick (after approximation) $\psi=\tilde{w}_R\eta^2$ as a test function in Definition \ref{defsolc}. So, it is easily seen that
\begin{equation}\label{e1h}
\begin{aligned}
& \mu_s \int_{\Omega_R} y^{1-2s} \left|\nabla \tilde w_R(x,y) \right|^2 \eta^2 \, \dd x \dd y \\
\leq & 2 \left\| w_R \right\|_\infty^{\alpha+1} \int_{B_R} \eta^2(x,0) \, \rho(x) \dd x + 4 \, \mu_s \left\| \tilde{w}_R \right\|^2_\infty \int_{\Omega_R}  y^{1-2s}  \left|\nabla \eta(x,y) \right|^2 \dd x \dd y \, .
\end{aligned}
\end{equation}
From \eqref{e89}, \eqref{190219}, \eqref{190219-bis} and \eqref{e1h} we deduce that, for any $ \Omega_0 \Subset \mathbb{R}^{d+1}_+ \cup \partial \mathbb{R}^{d+1}_+ $, there holds
\begin{equation}\label{e2h}
\int_{\Omega_0} y^{1-2s} \left| \nabla \tilde{w}_R(x,y) \right|^2 \, \dd x \dd y \le K
\end{equation}
for a suitable positive constant $K$ independent of $R>0$. By collecting \eqref{e89}, \eqref{190219}, \eqref{190219-bis} and \eqref{e90}, we infer that there exist the following (nontrivial) pointwise limits:
\begin{equation}\label{e92}
\lim_{R \to \infty} \tilde w_R = {\tilde{w}} \in L^\infty(\mathbb{R}^{d+1}_+) \, ,  \quad \lim_{R \to \infty} w_R =  {w} \in L^\infty(\mathbb{R}^d) \, .
\end{equation}
Due to \eqref{e2h}, by standard compactness arguments we can pass to the limit in the weak formulation \eqref{03011-4} and infer that $ {w} $ is a local weak solution to \eqref{e20d} in the sense of Definition \ref{defsol01} (with local extension ${\tilde{w}}$).

Now we have to prove minimality. Hereafter, we shall denote by $ \underline{w} $ the solution constructed above and by $w$ any other nonidendically zero local weak solution to \eqref{e20d} (according to Definition \ref{defsol01}). In particular, for $R$ large enough $(w|_{B_R},\tilde{w}|_{\Omega_R})$ is a nontrivial solution to problem \eqref{e20c}, in the sense of Definition \ref{defsolc}, except that $\tilde{w}|_{\Omega_R}$ is not necessarily zero on $\Sigma_R$ (that is, $\tilde{w}$ has finite energy in $\Omega_R$ but does not belong to $X^s_0(\Omega_R)$). However, the regularity results of \cite{CS1} hold regardless of boundary conditions on $ \Sigma_R $: namely, Proposition \ref{prop-CS1}-\eqref{st-2} is applicable in this case as well, ensuring that $ w > 0 $ in $B_R$. Because $(w_R,\tilde{w}_R)$ is also a weak solution to \eqref{e20c} and, trivially, $\tilde{w}_R|_{\Sigma_R} \le \tilde{w}|_{\Sigma_R}$ on $\Sigma_R$, thanks to Proposition \ref{prop-CS1}-\eqref{st-1} (with the choices $(w_{R}^{(1)}, \tilde{w}^{(1)}_{R})=(w_R,\tilde{w}_R)$ and $(w_{R}^{(2)}, \tilde{w}^{(2)}_{R})=(w|_{B_R}, \tilde{w}|_{\Omega_R})$) we deduce
\[
w_R \leq  w|_{B_R}  \quad \textrm{in }  B_R \, ,
\]
whence $\underline{w}\le w $ in $\Gamma $ by letting $R \to \infty $, so that $\underline{w}$ is indeed minimal. The bound \eqref{e59} is then just a consequence of \eqref{e89}, \eqref{190219-bis} and \eqref{e92}.

From the above method of proof one can check that, under the more restrictive condition \eqref{eq: cond-ener}, then $\underline{w}$ is also a \emph{weak solution} to \eqref{19021} in the sense of Definition \ref{defsol01-weak-energy}. In fact, thanks to Lemma \ref{l100}, the inequalities \eqref{e89}, \eqref{190219-bis} and condition \eqref{eq: cond-ener} ensure that $\{ \| w_R^{\alpha+1} \|_{1,\rho} \}$ is uniformly bounded with respect to $R$. As a consequence, it is easy to verify that estimate \eqref{e2h} holds with $\Omega_0=\mathbb{R}^{d+1}_+ $ (up to setting $\tilde{w}_R= 0$ in $\Omega_R^c$). By passing to the limit as $R\to\infty $, this implies that $\underline{\tilde{w}} \in {X}^s$, $\underline{w} \in \dot{H}^s(\mathbb{R}^d)$, $\underline{\tilde{w}}=\D{E}(\underline{w}) $ and $\underline{w}$ satisfies \eqref{weak-ener}.

As already remarked, the fact that $\underline{w}$ is a \emph{very weak solution} to \eqref{19021} in the sense of Definition \ref{defsol01-very-weak} for all $\gamma>2s$ will be deduced at the end of the asymptotic analysis of Section \ref{as2} (see the proof of Theorem \ref{thmab2}).
\end{proof}

\subsection{Uniqueness}
In this section we prove our uniqueness result, stated in Theorem \ref{02034}, for solutions to \eqref{19021}. The strategy of proof strongly relies on the uniqueness result provided by Theorem \ref{prthm6} for solutions to \eqref{eq: regular}.
\begin{proof}[Proof of Theorem \ref{02034}]
Set $m= 1/\alpha$ and
\[
C_m = (m-1)^{-\frac{1}{m-1}} \, .
\]
For any $k \in \mathbb{N}$ let $\zeta_k\in C^\infty(\mathbb R^d)$ be such that $\zeta_k = 1$ in $B_k$, $\zeta_k = 0$ in $B^c_{2k}$ and $0\leq  \zeta_k \leq 1$ in $\mathbb{R}^d$. Take $R>2k$ and denote as $(v_{R,k}, \tilde{v}_{R,k})$ the unique strong solution to the following evolution problem (see Appendix \ref{rd} -- Part II):
\begin{equation}\label{e68}
\begin{cases}
L_s (\tilde{v}^m_{R, k} )= 0  & {\rm in } \ \Omega_R \times(0,\infty) \, , \\
\tilde v_{R, k} = 0 & {\rm on } \ \Sigma_R \times(0,\infty) \, , \\
\displaystyle \frac{\partial (\tilde{v}^m_{R, k})}{\partial y^{2s}} = \rho \displaystyle \frac{\partial v_{R, k}}{\partial t} & {\rm on } \  \Gamma_R \times(0,\infty) \, , \\
v_{R,k}= C_m \, \zeta_k w^{\frac1m} &  {\rm on } \ B_R \times \{t=0 \} \, .
\end{cases}
\end{equation}
Let $(w_R, \tilde w_R)$ be defined as in the proof of Theorem \ref{prthm8}. Since by hypothesis $w \in L^\infty(\mathbb{R}^d)$, thanks to \eqref{17042-bis} we can select a suitable $\tau_R>0$ so that
\begin{equation}\label{eq: dato-iniz-compare}
\frac{w_{R+1}^\frac{1}{m}}{\tau_R^{\frac{1}{m-1}}} \ge {w}^\frac{1}{m} \quad \textrm{in } B_R \, .
\end{equation}
We have:
\begin{equation}\label{eq: prima-dis}
\widetilde{U}_{R} = \frac{C_m{\tilde w_{R+1}}^\frac{1}{m}}{(t+\tau_R)^{\frac{1}{m-1}}} \leq \frac{C_m {\tilde w}_{R+1}^\frac{1}{m}}{t^{\frac{1}{m-1}}} = \widetilde{U}_{0R} \quad \textrm{in } \overline{\Omega}_R  \times (0, \infty) \, .
\end{equation}
Set ${U}_{R}(\cdot,t)=\widetilde{U}_{R}(\cdot,0,t)$ and ${U}_{0R}(\cdot,t)=\widetilde{U}_{0R}(\cdot,0,t)$, for each $t>0$. By definition of $({U}_{R},\widetilde{U}_{R})$ and recalling \eqref{eq: dato-iniz-compare}, we get that $({U}_{R},\widetilde{U}_{R})$ is a strong supersolution to \eqref{e68}. Hence, by the comparison principle stated in Proposition \ref{prop1cl} below and \eqref{eq: prima-dis}, we deduce:
\begin{equation}\label{04032}
{v}_{R,k} \leq {U}_{R} \leq {U}_{0R} \quad \textrm{a.e.\ in } B_R \times (0,\infty) \, .
\end{equation}
In addition to the above bounds we also have that, for any $k_2>k_1$ and $R>2 k_1$, there holds
\begin{equation}\label{04032-bis}
{v}_{R,k_1} \leq {v}_{R,k_2} \leq \frac{C_m {w}^\frac{1}{m}}{\left(t+1\right)^{\frac{1}{m-1}}} = V  \quad \textrm{a.e.\ in } B_R \times (0,\infty) \, .
\end{equation}
Such inequalities follow by noticing that $(V,\widetilde{V})$ is a strong supersolution to \eqref{e68} for all $R>0$ and $k \in \mathbb{N}$, while $({v}_{R,k_2},\tilde{v}_{R,k_2})$ is a strong supersolution to \eqref{e68} for $k=k_1$. One then applies again Proposition \ref{prop1cl}.

Since for each $k \in \mathbb{N} $ we have $C_m \zeta_k w^{\frac1m} \in L^1_\rho(\mathbb R^d)\cap L^\infty(\mathbb R^d)$, by standard arguments (e.g.\ similar to the ones
exploited in the proof of \cite[Theorem 3.1]{PT1}, see also Appendix \ref{rd} -- Part II) one sees that there exists the limit
\[
v_{\infty, k} = \lim_{R\to \infty} v_{R, k} \quad \textrm{a.e.\ in } \mathbb{R}^d
\]
and it is a solution of the problem
\[
\begin{cases}
\rho \, (v_{\infty,k})_t + (- \Delta)^{s}( v_{\infty,k}^m ) = 0  & \textrm{in } {\mathbb R}^d \times (0,\infty) \, , \\
v_{\infty,k}= C_m \, \zeta_k w^{\frac1m} & \textrm{on } \mathbb{R}^d \times \{ 0 \}  \, ,
\end{cases}
\]
both in the sense of Definition \ref{defsol1} and in the sense of Definition \ref{defsol05} below. Moreover, as a consequence of \eqref{04032-bis}, such limit satisfies the bounds
\begin{equation}\label{04032-ter}
{v}_{\infty,k_1} \le {v}_{\infty,k_2} \leq V  \quad \textrm{a.e.\ in } \mathbb{R}^d \times (0,\infty)
\end{equation}
for all $k_2 > k_1$. Thanks to \eqref{04032-ter} we get the existence of the pointwise limit
\begin{equation}\label{04032-qua}
v_{\infty} = \lim_{k\to \infty} v_{\infty, k}  \le V  \quad \textrm{a.e.\ in } \mathbb{R}^d \times (0,\infty) \, ;
\end{equation}
by passing to the limit in the very weak formulation solved by $v_{\infty,k}$ for all $k\in \mathbb N$, we infer that $v_\infty $ is a very weak solution, in the sense of Definition \ref{defsol05}, to the problem
\begin{equation}\label{04033b}
\begin{cases}
\rho\, (v_\infty)_t + (- \Delta)^{s}( v_\infty^m) = 0  & \textrm{in } {\mathbb R}^d\times(0,\infty) \, , \\
v_\infty = C_m w^\frac{1}{m} & \textrm{on } \mathbb{R}^d \times \{0\}  \, .
\end{cases}
\end{equation}
Now notice that $V$ is also a very weak solution to \eqref{04033b}. Because, by hypothesis, $w \in L^1_{(1+|x|)^{-\beta}}(\mathbb{R}^d)$, clearly $V^m \in L^1_{(1+|x|)^{-\beta}}(\mathbb{R}^d \times (0,T) )$. Hence, thanks to \eqref{04032-qua}, we deduce that also $v_\infty^m $ belongs to $L^1_{(1+|x|)^{-\beta}}(\mathbb{R}^d \times (0,T) )$. We are therefore in position to apply Theorem \ref{prthm6} (after Remark \ref{oss101}) and obtain
$$ v_\infty = V \quad \textrm{a.e.\ in } \mathbb{R}^d \times (0,\infty) \, . $$
Passing to the limit in \eqref{04032} (first as $R\to\infty$, then as $k\to\infty$) and using \eqref{e92}, we infer that
\[
v_\infty \le \frac{C_m {\underline{w}}^{\frac{1}{m}}}{t^{\frac{1}{m-1}}} \quad\textrm{a.e. in } \mathbb{R}^d \times (0, \infty) \, ,
\]
Hence,
\begin{equation}\label{eq: quasi-last-uniq}
\frac{w^\frac{1}{m}}{\underline{w}^\frac{1}{m}} \leq \frac{(t+1)^\frac{1}{m-1}}{t^\frac{1}{m-1}} \, ,
\end{equation}
and by letting $t \to \infty$ in \eqref{eq: quasi-last-uniq} we deduce
\[
w^\frac{1}{m} \leq \underline{w}^\frac{1}{m} \quad \textrm{a.e. in } \mathbb{R}^d  \, .
\]
Since $w$ is nontrivial and $\underline{w}$ is minimal, it follows that $w = \underline{w}$.
\end{proof}

\section{Asymptotic behaviour for rapidly decaying densities: proofs}\label{as2}
Before proving Theorem \ref{thmab2}, we need the following intermediate result, which gives a crucial bound from above for the solution to problem \eqref{eq: regular} provided by Theorem \ref{prthm6}.
\begin{lem}\label{lemab}
Under the same assumptions and with the same notations as in Theorem \ref{thmab2}, there holds
\begin{equation}\label{e67}
u(x,t) \leq (m-1)^{-\frac{1}{m-1}} \, t^{-\frac 1{m-1}}{w}^{\frac 1 m}(x) \quad \textrm{for a.e. } (x,t) \in \mathbb{R}^d \times (0,\infty) \, .
\end{equation}
\end{lem}
We point out that an analogous result was proved in \cite{BV2} in bounded domains for $\rho\equiv 1$. Such result has then been applied in the study of the asymptotic behaviour of solutions in \cite{BSV}.

\begin{proof}
Suppose at first that $u_0\in L^1_\rho(\mathbb R^d)\cap L^\infty(\mathbb R^d)$. Let $C_m$, $(w_R, \tilde{w}_R)$ and $(U_R,\widetilde{U}_R)$ (for a suitable $\tau_R>0$ to be chosen later) be defined as in the proofs of Theorems \ref{prthm8} and \ref{02034}. For any $R>0$, let $(u_R, \tilde{u}_R)$ be the unique strong solution to the following evolution problem (see Appendix \ref{rd} -- Part II):
\begin{equation}\label{e70}
\begin{cases}
L_s ( \tilde{u}_R^m ) = 0  & {\rm in} \ \Omega_R \times(0,\infty) \, , \\
\tilde u_R=0 & {\rm on } \ \Sigma_R\times(0,\infty) \, , \\
 \displaystyle \frac{\partial \tilde{u}^m_R}{\partial y^{2s}} = \rho\displaystyle\frac{\partial u_R }{\partial t} & {\rm on } \  \Gamma_R \times (0,\infty) \, ,  \\
u_R = u_0  &  {\rm on } \ B_R \times \{ t=0 \} \, .
\end{cases}
\end{equation}
By standard arguments (see again the proof of \cite[Theorem 3.1]{PT1} and Appendix \ref{rd} -- Part II), we have that
\begin{equation}\label{e75}
\lim_{R \to \infty} u_R = u \quad \textrm{a.e.\ in } \mathbb{R}^d\times (0,\infty) \, , \quad \lim_{R \to \infty} \tilde{u}_R^m =\tilde{u}^m = \D{E}(u^m) \quad \textrm{a.e.\ in } \mathbb{R}^{d+1}_+ \times (0,\infty) \, ,
\end{equation}
where $u $ is the solution to \eqref{eq: regular} provided by Proposition \ref{prop2c}. Note that, thanks to \eqref{17042-bis}, for any $R>0$ there holds
\begin{equation}\label{eq: rel-inf}
\min_{\overline{B}_R} w_{R+1} > 0  \, .
\end{equation}
Hence, in view of \eqref{eq: rel-inf} and recalling that we assumed $u_0 \in L^1_\rho(\mathbb R^d)\cap L^\infty(\mathbb R^d)$, we can pick $\tau_R>0$ so that
\begin{equation}\label{e72}
\frac{C_m w_{R+1}^\frac{1}{m}}{\tau_R^{\frac{1}{m-1}}} \ge u_0
\quad\textrm{a.e.\ in } B_R \, .
\end{equation}
Due to \eqref{e72}, $(U_R, \widetilde{U}_R)$ is a strong supersolution to problem \eqref{e70}. Therefore, by comparison principles (see Proposition \ref{prop1cl} below),
\begin{equation}\label{e73}
u_R \leq U_R \quad \textrm{a.e.\ in } B_R \times (0,\infty) \, .
\end{equation}
Because trivially $U_R \le C_m t^{-\frac1{m-1}}w_{R+1}^{\frac 1 m}$, from \eqref{e73} we deduce the fundamental estimate
\begin{equation}\label{e74}
u_R \leq C_m \, t^{-\frac 1{m-1}} w_{R+1}^{\frac1m} \quad \textrm{a.e.\ in } B_R \times (0,\infty) \, .
\end{equation}
By letting $R \to \infty$ in \eqref{e74} and recalling \eqref{e92} and \eqref{e75}, we finally get \eqref{e67}.

Consider now general data $u_0 \in L^1_{\rho}(\mathbb{R}^d)$. In this case, we have that
$$ u =\lim_{n \to \infty} u_n \quad \textrm{a.e.\ in } \mathbb{R}^d \times (0,\infty) \, , $$
where for every $n \in \mathbb{N}$ we denote as $u_n$ the solution to problem \eqref{eq: regular} corresponding to the initial datum $u_{0n} \in L^1_\rho(\mathbb{R}^d)\cap L^\infty(\mathbb R^d)$, and the sequence $\{ u_{0n} \}$ is such that $0\leq u_{0n} \le u_0 $ in $\mathbb{R}^d $ for all $n \in \mathbb{N}$ and $ u_{0n}\to u_0$ in $L^1_\rho(\mathbb R^d)$ as $n\to\infty$ (see \cite[Section 6.2]{PT1} and Appendix \ref{rd} -- Parts I, II ). In view of the first part of the proof, we know that for every $n \in \mathbb{N}$ there holds
\begin{equation}\label{eq: lim-n}
u_n \leq C_m \, t^{-\frac 1{m-1}} {w}^{\frac 1 m} \quad \textrm{a.e.\ in } \mathbb{R}^d \times (0,\infty) \, .
\end{equation}
The assertion then follows by passing to the limit as $n\to \infty$ in \eqref{eq: lim-n}.
\end{proof}

\begin{oss}\label{remark-ext} \rm
As a consequence of the method of proof of Lemma \ref{lemab} we also get the validity of the estimate
\begin{equation}\label{e3h}
\D{E}(u^m) \leq  C_m^m \, t^{-\frac m{m-1}} \, \tilde{w}  \quad \textrm{a.e.\ in } \mathbb{R}^{d+1}_+ \times (0,\infty) \, ,
\end{equation}
where $\D{E}(u^m)$ is the extension of $u^m$ (see the beginning of Section \ref{subl}) and $\tilde{w}$ is the local extension of $w$, in agreement with Definition \ref{defsol01}, provided along the first part of the proof of Theorem \ref{prthm8}. In fact it is enough to notice that, by standard comparison principles for sub- and supersolutions to the problem $L_s=0$ in $\Omega_R$, from \eqref{e74} it follows that
\begin{equation}\label{e74-estese}
\tilde{u}_R^m \leq C_m^m \, t^{-\frac m{m-1}} \tilde{w}_{R+1} \quad \textrm{a.e.\ in } \Omega_R \times (0,\infty) \, ,
\end{equation}
whence \eqref{e3h} upon letting $R\to\infty$ in \eqref{e74-estese}.
\end{oss}

\begin{proof}[Proof of Theorem \ref{thmab2} and end of proof of Theorem \ref{prthm8}]
Let us denote as $v(x,\tau)$ the following rescaling of $u(x,t)$:
\begin{equation}\label{eq: def-v}
u(x,t)= e^{-\beta \tau} v(x, \tau) \, , \quad t= e^{\tau} \, , \quad \beta=\frac 1 {m-1} \, .
\end{equation}
It is immediate to check that $v$ is a (weak, and in particular very weak) solution to the equation
\[
\rho v_\tau = -(-\Delta)^s(v^m) + \beta \rho v \quad \textrm{in } \mathbb{R}^d \times(0,\infty) \, ,
\]
in the sense that
\begin{equation}\label{e81}
\begin{aligned}
& - \int_0^\infty \int_{\mathbb{R}^d}  v(x,\tau) \varphi_\tau(x,\tau) \, \rho(x) {\rm d}x {\rm d}\tau +\int_0^\infty \int_{\mathbb{R}^d} v^m(x,\tau) (-\Delta)^s(\varphi)(x,\tau) \, {\rm d}x {\rm d}\tau \\
= & \beta\int_0^\infty \int_{\mathbb{R}^d}  v(x,\tau) \varphi(x,\tau) \, \rho(x) {\rm d}x{\rm d}\tau + \int_{\mathbb{R}^d} u(x,1) \varphi(x,0) \, \rho(x) {\rm d}x
\end{aligned}
\end{equation}
for all $\varphi \in C^\infty_c(\mathbb{R}^d \times [0,\infty) )$. Moreover, $\D{E}(v^m) \in L^2_{\rm loc}((0,\infty); X^s)$ and
\begin{equation}\label{e81b}
\begin{aligned}
& - \int_0^T \int_{\mathbb{R}^d} v(x,\tau) {\psi}_\tau(x,0,\tau) \, \rho(x) {\rm d}x {\rm d}\tau + \mu_s \int_0^T \int_{\mathbb{R}^{d+1}_+} y^{1-2s} \left \langle \nabla{\D{E}(v^m)} , \nabla {\psi} \right\rangle(x,y,\tau) \, {\rm d}x {\rm d}y {\rm d}\tau \\
= & \beta \int_0^T \int_{\mathbb{R}^d}  v(x,\tau) {\psi}(x,0,\tau) \, \rho(x) {\rm d}x{\rm d}\tau
\end{aligned}
\end{equation}
for all $T>0$ and $ {\psi} \in C^\infty_c((\mathbb{R}^{d+1}_+ \cup \partial \mathbb{R}^{d+1}_+) \times (0,T))$. Thanks to Lemma \ref{lemab}, we have:
\begin{equation}\label{e78}
v(x,\tau) \leq C_m  w^{\frac 1 m}(x) \leq C_m \left\| w \right\|_{\infty}^{\frac 1 m} \quad \textrm{for a.e.\ } (x,\tau) \in \mathbb{R}^d \times (0,\infty) \, ;
\end{equation}
furthermore, recalling Remark \ref{remark-ext},
\begin{equation}\label{e78b}
\D{E}(v^m)(x, y, \tau) \leq C_m^m \tilde{w}(x,y) \leq C_m^m  \| \tilde{w} \|_{\infty} \quad \textrm{for a.e.\ } (x,y,\tau) \in \mathbb{R}^{d+1}_+ \times (0,\infty) \, .
\end{equation}
Now let us show that
\begin{equation}\label{e1l}
v(x,\tau_2) \ge v(x,\tau_1) \quad \textrm{for a.e. } x \in \mathbb{R}^d \, , \quad  \D{E}(v^m)(x,y,\tau_2) \ge \D{E}(v^m)(x,y,\tau_1) \quad \textrm{for a.e. } (x,y) \in \mathbb{R}^{d+1}_+
\end{equation}
for all $\tau_2 \ge \tau_1 >0 $. To this purpose, first of all note that, similarly  to \cite[p.\ 182]{Vaz07} (see also the original reference \cite{BC}), one can
prove the fundamental B\'enilan-Crandall inequality
\[
\rho u_t \ge - \frac{\rho u}{(m-1) t} \quad \textrm{a.e.\ in }
\mathbb{R}^d \times (0,\infty)
\]
which, recalling \eqref{eq: def-v}, implies that
\begin{equation}\label{e79}
v_\tau \geq 0 \quad \textrm{a.e.\ in } \mathbb{R}^d \times (0,\infty) \, .
\end{equation}
Thanks to \eqref{e79} we obtain the first inequality in
\eqref{e1l}, and therefore also the second one because the
extension operator is order preserving. Hence, by \eqref{e78},
\eqref{e78b} and \eqref{e1l} we infer that there exist finite the
limits
\begin{equation}\label{e106} h(x)
= \lim_{k \to \infty} v(x,\tau_k)  \quad \textrm{for a.e.\ } x \in
\mathbb{R}^d  \, , \quad H(x,y) = \lim_{k \to \infty} \D{E}
(v^m)(x,y,\tau_k) \quad \textrm{for a.e.\ } (x,y) \in
\mathbb{R}^{d+1}_+ \,,
\end{equation}
where $\{\tau_k\}$ is any time sequence tending to infinity. Moreover, since $u_0 \not \equiv
0$, \eqref{e1l} implies that $ h \not \equiv 0 $ and $ H\not\equiv
0 $, while \eqref{e78} and \eqref{e78b} ensure that $h \in
L^\infty(\mathbb{R}^d)$ and $ H \in L^\infty(\mathbb{R}^{d+1}_+)$.

Let us set
\begin{equation}\label{eq: def-funct-g}
g= C_m^{-m} h^m \, , \quad \tilde{g}= C_m^{-m} H \, .
\end{equation}
First we want to prove that $g$ (with the corresponding local extension $\tilde{g}$) is a solution to problem \eqref{e20d} (for $\alpha=1/m $) in the sense of Definition \ref{defsol01}. To this end, for any fixed $0 < \tau_1 < \tau_2 $ and $ 0<\epsilon<(\tau_2-\tau_1)/{2}$, let $\zeta_\epsilon(\tau) $ be a smooth approximation of
the function $\chi_{[\tau_1, \tau_2]}(\tau)$ such that
\[
0 \leq \zeta_\epsilon (\tau) \leq 1 \quad \forall \tau \ge 0 \, , \quad \zeta_\epsilon(\tau) = 0 \quad \forall \tau \not \in [\tau_1, \tau_2] \, , \quad \zeta_\epsilon(\tau) = 1  \quad \forall \tau \in [\tau_1+\epsilon, \tau_2-\epsilon] \, .
\]
Furthermore, we can and shall assume that
\[
\zeta_\epsilon^\prime (\tau) \to \delta(\tau-\tau_1) - \delta(\tau-\tau_2)
\]
as $\epsilon \to 0$. Consider now a cut-off function $\eta$ as in the first part of the proof of Theorem \ref{prthm8} and plug in the weak formulation \eqref{e81b} the test function ${\psi}=\zeta_\epsilon \eta^2 \D{E}(v^m)  $. Upon letting $\epsilon \to 0 $, we get:
\[
\begin{aligned}
& \frac{1}{m+1} \int_{\mathbb{R}^d} v^{m+1}(x,\tau_2) \, \eta^2(x,0) \, \rho(x) {\rm d}x \! + \! \mu_s \int_{\tau_1}^{\tau_2} \int_{\mathbb{R}^{d+1}_+} y^{1-2s} \left \langle \nabla{\D{E}(v^m)} , \nabla [\eta^2 \, \D{E}(v^m)] \right\rangle \! (x,y,\tau) \, {\rm d}x {\rm d}y {\rm d}\tau \\
= & \frac{1}{m+1} \int_{\mathbb{R}^d} v^{m+1}(x,\tau_1) \, \eta^2(x,0) \, \rho(x) {\rm d}x + \beta \int_{\tau_1}^{\tau_2} \int_{\mathbb{R}^d}  v^{m+1}(x,\tau) \, \eta^2(x,0) \, \rho(x) {\rm d}x{\rm d}\tau \, .
\end{aligned}
\]
Thanks to \eqref{e78} and \eqref{e78b}, by setting $\tau_1=\tau_k
$, $\tau_2=\tau_{k}+1 $ and proceeding as in the proof of
\eqref{e2h}, we obtain the estimate
\begin{equation}\label{e30c-energia-lic}
\int_{\tau_k}^{\tau_k+1} \int_{\Omega_0} y^{1-2s} \left| \nabla
\D{E}(v^m)(x,y,\tau)\right|^2 \dd x \dd y \dd \tau \leq C^\prime
\end{equation}
for any $\Omega_0 \Subset  {\mathbb{R}^{d+1}_+ \cup \partial \mathbb{R}^{d+1}_+} $ and a suitable constant $C^\prime>0$ independent of $k$. Take any function $\phi \in C_c^\infty(\mathbb{R}^{d+1}_+ \cup \partial \mathbb{R}^{d+1}_+)$. By plugging in \eqref{e81b} the test function ${\psi}(x,y,\tau)=\phi(x,y)\zeta_\epsilon(\tau)$ and letting $\epsilon \to 0$, we infer that
\begin{equation}\label{e81b-changed}
\begin{aligned}
&  \int_{\mathbb{R}^d} \left[ v(x,\tau_k+1)-v(x,\tau_k) \right]  \phi(x,0) \, \rho(x) {\rm d}x \\
 & + \mu_s \int_{\tau_k}^{\tau_k+1} \int_{\mathbb{R}^{d+1}_+}  y^{1-2s} \left \langle \nabla{\D{E}(v^m)}(x,y,\tau) , \nabla {\phi}(x,y) \right\rangle {\rm d}x {\rm d}y {\rm d}\tau \\
= & \beta \int_{\tau_k}^{\tau_k+1} \int_{\mathbb{R}^d} v(x,\tau)
\phi(x,0) \, \rho(x) {\rm d}x{\rm d}\tau \, .
\end{aligned}
\end{equation}
We point out that, still as a consequence of the time monotonicity
ensured by \eqref{e1l}, in addition to \eqref{e106} we also have
\begin{gather}\nonumber
h(x) = \lim_{k \to \infty} v(x,\tau_k+\lambda) \quad \textrm{for a.e.\ } (x,\lambda) \in \mathbb{R}^d \times (0,1) \, , \\
h(x) = \lim_{k \to \infty} v(x,\tau_k+1) \quad \textrm{for a.e.\ } x \in \mathbb{R}^d \, , \label{e106-general} \\
H(x,y) = \lim_{k \to \infty} \D{E} (v^m)(x,y,\tau_k+\lambda) \quad
\textrm{for a.e.\ } ((x,y),\lambda) \in \mathbb{R}^{d+1}_+ \times
(0,1) \, . \nonumber
\end{gather}
Gathering \eqref{e78}, \eqref{e78b}, \eqref{e106},
\eqref{e30c-energia-lic} and \eqref{e106-general}, we can pass to
the limit safely in \eqref{e81b-changed} to find that $h$ and $H$
satisfy
\[
\mu_s \int_{\mathbb{R}^{d+1}_+}  y^{1-2s} \langle \nabla{H} ,
\nabla {\phi} \rangle(x,y) \, {\rm d}x {\rm d}y = \beta
\int_{\mathbb{R}^d} h(x) \phi(x,0) \, \rho(x) {\rm d}x \, ,
\]
with $H(x,0)=h^m(x)$. That is, the function $g$ (with $\tilde{g}$
as a local extension) defined in \eqref{eq: def-funct-g} is a
local weak solution to \eqref{e20d} (for $\alpha=1/m$) in the sense
of Definition \ref{defsol01}. Furthermore, $g$ is also a very weak
solution to \eqref{19021} in the sense of Definition
\ref{defsol01-very-weak}. In order to prove the latter assertion,
we can proceed as above: for any $\phi \in
C^\infty_c(\mathbb{R}^d)$ plug in the weak formulation \eqref{e81}
the test function $\varphi(x,\tau)=\zeta_\epsilon(\tau) \phi(x)$
and let $\epsilon \to 0 $ to get
\begin{equation*}\label{e82}
\begin{aligned}
& \int_{\mathbb{R}^d}[v(x,\tau_k+1)-v(x,\tau_k)] \phi(x) \, \rho(x) {\rm d}x \\
= & \int_{\tau_k}^{\tau_k+1} \int_{\mathbb{R}^d}
\left[-v^m(x,\tau)(-\Delta)^s(\phi)(x) + \beta v(x,\tau) \phi(x)
\rho(x) \right] {\rm d}x {\rm d}\tau \, .
\end{aligned}
\end{equation*}
Passing to the limit as $ k \to \infty $ and using \eqref{e78},
\eqref{e106}, \eqref{e106-general}, we end up with
\begin{equation}\label{e83}
0= -\int_{\mathbb{R}^d} g(x) (-\Delta)^s(\phi)(x) \, {\rm d}x +
\int_{\mathbb{R}^d} g^{\frac1m}(x) \phi(x) \, \rho(x) {\rm d}x
\end{equation}
and the estimate
\begin{equation}\label{e78-limit}
g(x) \leq {w}(x) \quad \textrm{for a.e.\ } x \in \mathbb{R}^d \, .
\end{equation}
Since $g$ is a non-identically zero local weak solution to
\eqref{e20d}, the minimality of $w$ and \eqref{e78-limit}
necessarily imply that $g = w$. In particular, thanks to
\eqref{e83}, we can conclude the proof of Theorem \ref{prthm8} by
inferring that the minimal solution provided by it is also a very
weak solution to \eqref{19021} (for $\alpha=1/m$) in the sense of
Definition \ref{defsol01-very-weak}.

Finally, the convergence of $ \{
v(\tau_k) \} $ to $ C_m w^{1/m} $ in $ L^p_{\rm loc}(\mathbb R^d)
$ for $ p \in [1,\infty) $ is just a consequence of \eqref{e78}
and \eqref{e106}. The above arguments being independent of the
particular sequence $ \{ \tau_k \} $, the proof is complete.
\end{proof}

\section{Asymptotic behaviour for slowly decaying densities: proofs}\label{as1}
In order to prove Theorem \ref{thmab1}, we first need some preliminary results. 

\begin{lem}\label{lem: potential-basic-prop-1}
Let $d>2s$ and assume that $ \rho $ satisfies \eqref{eq: ass-rho-slow-bis} for some $ \gamma \ge 0 $ and $ \gamma_0 \in [0,2s) $. Let $ U_\rho^v $ be the Riesz potential of $ v \in L^1_{ \rho }(\mathbb{R}^d) \cap L^\infty(\mathbb{R}^d) $, that is
\[
U^{v}_\rho = I_{2s} \ast\left( \rho v \right) .
\]
Then $U^{v}_\rho $ belongs to $ C(\mathbb{R}^d) \cap L^p(\mathbb{R}^d) $ for all $p$ satisfying
\begin{equation}\label{eq: potential-v-Lp}
p \in \left( \frac{d}{d-2s} , \infty \right] .
\end{equation}
\end{lem}
\begin{proof}
We refer the reader to \cite[Lemma 4.8]{GMP2}.
\end{proof}

Let $u$ be a weak solution to problem \eqref{eq:
measure}, according to Definition \ref{defsol1}. For any
$\lambda>0$, set
\begin{equation}\label{eq: def-sol-riscalata}
u_{\lambda}(x,t)=\lambda^\alpha u (\lambda^\kappa x, \lambda t)
\quad \forall (x,t) \in \mathbb{R}^d \times (0,\infty) \, ,
\end{equation}
where $\alpha,\kappa $ are defined in \eqref{alpha-kappa}. Notice
that \eqref{eq: def-sol-riscalata} is the same scaling under which
$ u^{c_\infty}_M $ is invariant (see Section \ref{sect:
sect-bar}).

\begin{pro}\label{prorisc} Let the assumptions of Theorem \ref{thmab1} hold true, with in addition $ u_0 \in L^1_\rho(\mathbb{R}^d) \cap L^\infty(\mathbb{R}^d) $.
Then, for any sequence $\lambda_n \to \infty $, $\{
u_{\lambda_{n}} \} $ converges to $u^{c_\infty}_M$ almost
everywhere in $\mathbb{R}^d \times (0,\infty) $ along
subsequences.
\end{pro}
\begin{proof}
For notational simplicity, and with no loss of generality, we shall put $c_\infty=1$. Here we shall not give a fully detailed proof, since the procedure
follows closely the one performed in the proof of \cite[Theorem
3.2]{GMP2}. To begin with, note that $u_\lambda$ solves the problem
\begin{equation}\label{e14}
\begin{cases}
 \rho_\lambda \, u_t + (-\Delta)^s (u_\lambda^m) = 0  &  \textrm{in } \mathbb{R}^d \times (0,\infty) \, , \\
 u_\lambda = u_{0\lambda}  &  \textrm{on } \mathbb{R}^d \times \{0\}  \, ,
\end{cases}
\end{equation}
where
\begin{equation}\label{eq: rho-lambda-u_0-lambda}
\rho_\lambda(x)=\lambda^{\kappa \gamma}\rho(\lambda^\kappa x) \, ,
\quad u_{0\lambda}(x)=\lambda^\alpha u_0(\lambda^\kappa x) \quad
\forall x\in \mathbb{R}^d \,.
\end{equation}
It is easily seen that (recall the conservation of mass \eqref{eq:
cons-mass})
\begin{equation}\label{e15}
\left\|u_{\lambda}(t)\right\|_{1,\rho_\lambda
}=\left\|u_{0\lambda}\right\|_{1,\rho_\lambda}=M \quad  \forall
t,\lambda>0 \, .
\end{equation}
\noindent {\bf Claim 1:} {\it There exists a subsequence
$\{u_{\lambda_m}\} \subset \{u_{\lambda_n}\}$ that converges
pointwise a.e.\ in $\mathbb{R}^d \times (0,\infty)$ to some
function $u$, which satisfies \eqref{e23}, \eqref{e24} and
\eqref{e25} with $ \rho(x)=|x|^{-\gamma} $}.

\smallskip

\noindent First of all observe that, in view of \eqref{eq: ass-rho-slow-bis},
\begin{equation}\label{eq: bounds-rho-pre}
\frac{c}{1+|x|^\gamma} \le  \rho_\lambda(x) \le
\frac{C}{|x|^\gamma} \quad \forall \lambda \ge 1 \, .
\end{equation}
Combining the smoothing effect \eqref{eq:
smoothing-effect-general} with \eqref{e15}, we obtain:
\begin{equation}\label{e18}
\|u_\lambda(t)\|_{\infty}  \le  K \, t^{-\alpha} M^{\beta} \quad
\forall t,\lambda >0 \, ,
\end{equation}
where $K>0$ is a constant depending only on $ m $, $\gamma$,
$d$, $ s $ and $C$. In particular,
\begin{equation}\label{e19}
\int_{\mathbb{R}^d} u_\lambda^{m+1}(x,t) \, \rho_\lambda(x)
\mathrm{d}x \leq K^m \, t^{-\alpha m}M^{\beta m+1} \quad  \forall
t, \lambda>0 \, .
\end{equation}
By \eqref{eq: prima-esistenza-energy-1-teorema} and \eqref{e19},
we infer that
\begin{equation}\label{e20}
\int_{t_1}^{t_2} \int_{\mathbb{R}^d} \left|(-\Delta)^{\frac s
2}(u_\lambda^m)(x,t)\right|^2 \mathrm{d}x \mathrm{d}t + \frac
1{m+1}\int_{\mathbb{R}^d}u_\lambda^{m+1}(x,t_2) \, \rho_\lambda(x)
\mathrm{d}x \le \frac{K^m}{m+1} \, t_1^{-\alpha m}M^{\beta m+1}
\end{equation}
for all $\lambda>0$ and all $t_2>t_1>0$. On the other hand, thanks
to \eqref{eq: prima-esistenza-energy-2-teorema},
\begin{equation}\label{e21}
\int_{t_1}^{t_2}\int_{\mathbb{R}^d}\left|(z_\lambda)_t(x,t)\right|^2
\rho_\lambda(x) \mathrm{d}x\mathrm{d}t \leq  C^\prime \quad
\forall t_2>t_1>0 \, , \ \ \forall \lambda>0 \, ,
\end{equation}
where $z_\lambda:=u_\lambda^{(m+1)/2}$ and $ C^\prime $ is another
positive constant depending on $m$, $\gamma$, $ d $, $s$, $t_1$, $
t_2 $, $K$, $M$ but independent of $\lambda$. In view of
\eqref{e15}--\eqref{e21}, by standard compactness arguments (see
the proof of Theorem \ref{thm: teorema-esistenza}) the sequence
$\{ u_{\lambda_n} \}$ admits a subsequence $\{u_{\lambda_m}\}$
converging a.e.\ in $\mathbb{R}^d\times (0,\infty) $ to some
function $u$ that complies with \eqref{e23} and \eqref{e24}.
Moreover, recalling the assumptions on $\rho$, we have that
\eqref{eq: bounds-rho-pre} holds true and
\begin{equation}\label{e26b}
\lim_{\lambda\to\infty}\rho_{\lambda}(x)=|x|^{-\gamma} \quad
\textrm{for a.e. } x \in \mathbb{R}^d \, .
\end{equation}
We are therefore allowed to pass to the limit in the weak
formulation solved by $ u_{\lambda_m} $ to find that $u$ also
satisfies \eqref{e25}, and Claim 1 is proved. In order to deal
with the initial trace of $ u $, it is convenient to introduce the
Riesz potential $ U_\lambda $ of $ \rho_\lambda u_\lambda $, that
is $ U_\lambda(x,t) := [I_{2s} \ast \rho_\lambda u_\lambda(t)](x)
$.

\smallskip

\noindent{\bf Claim 2:} \emph{For any $\lambda>0$, the function
$U_\lambda$ satisfies
\begin{equation}\label{eq: u-tempo-lambda}
\int_{\mathbb{R}^d} U_\lambda(x,t_2) \phi(x) \, \mathrm{d}x -
\int_{\mathbb{R}^d} U_\lambda(x,t_1) \phi(x) \, \mathrm{d}x = -
\int_{\mathbb{R}^d} \left(\int_{t_1}^{t_2} u^m_\lambda(x,t) \,
\mathrm{d}t \right) \phi(x) \, \mathrm{d}x
\end{equation}
for all $ t_2>t_1>0 $ and $ \phi \in \mathcal{D}(\mathbb{R}^d) $.}

\smallskip

\noindent In order to prove \eqref{eq: u-tempo-lambda} rigorously,
one proceeds exactly as in the proof of \cite[Theorem 3.2]{GMP2}.
Note however that, formally, $
(-\Delta)^s(U_\lambda)(t)=\rho_\lambda u_\lambda(t)$, so that
\eqref{eq: u-tempo-lambda}, still at a formal level, just follows
by applying the operator $(-\Delta)^{-s} $ to both sides of the
differential equation in \eqref{e14}.

\smallskip

\noindent{\bf Claim 3:} \emph{For any $ \lambda>0 $, let
$U_{0\lambda}:= I_{2s} \ast \rho_\lambda u_{0\lambda} $. Then,
\begin{equation}\label{eq: u-tempo-lambda-bis}
\left| \int_{\mathbb{R}^d} U_\lambda(x,t_2) \phi(x) \, \mathrm{d}x
- \int_{\mathbb{R}^d} U_{0\lambda}(x) \phi(x) \, \mathrm{d}x
\right| \le \left\| \rho_\lambda^{-1} \phi \right\|_\infty
K^{m-1} \, M^{1+\beta(m-1)} \int_{0}^{t_2} t^{-\alpha(m-1)}
\mathrm{d}t
\end{equation}
for all $ t_2>0 $ and $ \phi \in \mathcal{D}(\mathbb{R}^d) $.}

\smallskip

\noindent The validity of \eqref{eq: u-tempo-lambda-bis} is just a
consequence of \eqref{e15}, \eqref{e18}, \eqref{eq:
u-tempo-lambda} and of the definition of weak solution.

\smallskip

\noindent{\bf Claim 4:} \emph{The potential $U$ of $ |x|^{-\gamma}
u$, that is $ U(x,t) := [I_{2s} \ast |\cdot|^{-\gamma}
u(\cdot,t)](x) $, satisfies
\begin{equation}\label{eq: u-tempo-lambda-ter}
\begin{aligned}
& \left| \int_{\mathbb{R}^d} U(x,t_2) \phi(x) \, \mathrm{d}x - \int_{\mathbb{R}^d} M I_{2s}(x) \phi(x) \, \mathrm{d}x \right| \\
\le & c^{-1} \left\| (1+|x|^\gamma) \phi \right\|_\infty  K^{m-1}
\, M^{1+\beta(m-1)} \int_{0}^{t_2} t^{-\alpha(m-1)} \mathrm{d}t
\end{aligned}
\end{equation}
for a.e.\ $ t_2>0 $ and $ \phi \in \mathcal{D}(\mathbb{R}^d) $.}

\smallskip

\noindent Our goal is to let $ \lambda \to \infty $ in \eqref{eq:
u-tempo-lambda-bis}. In the r.h.s.\ we just exploit \eqref{eq:
bounds-rho-pre}. Thanks to Claim 1, \eqref{e15}, \eqref{eq:
bounds-rho-pre}, \eqref{e18} and \eqref{e26b} we infer that
\[
\lim_{m \to \infty} \rho_{\lambda_m} u_{\lambda_m}(t) =
|x|^{-\gamma} u(t) \ \ \ \textrm{in } \sigma(\mathcal{M}
(\mathbb{R}^d),C_0(\mathbb{R}^d))
\]
for a.e.\ $t>0$. This is enough in order to pass to the limit in
the first integral in the l.h.s.\ of \eqref{eq:
u-tempo-lambda-bis}. The same holds true for the second integral,
provided we can prove that $ \{ \rho_\lambda u_{0\lambda} \} $
tends to $ M \delta $ e.g.~in $
\sigma(\mathcal{M}(\mathbb{R}^d),C_b(\mathbb{R}^d)) $ as $\lambda
\to \infty $. This is indeed the case: in fact, $\|\rho_\lambda
u_{0\lambda} \|_{1}=M$ and for any $ \phi \in C_c(\mathbb{R}^d) $
one has
\[
\begin{aligned}
& \lim_{\lambda\to\infty} \int_{\mathbb{R}^d} \phi(x) \rho_\lambda(x) u_{0\lambda}(x) \, \mathrm{d}x \\
= & \lim_{\lambda\to\infty} \lambda^{\alpha+\kappa \gamma}
\int_{\mathbb{R}^d} \phi(x) \rho(\lambda^\kappa x)
u_0(\lambda^\kappa x) \, \mathrm{d}x = \lim_{\lambda\to\infty}
\int_{\mathbb{R}^d} \phi({y}/{\lambda^\kappa}) u_0(y) \,  \rho(y)
\mathrm{d}y = M \phi(0) \, .
\end{aligned}
\]

\smallskip

\noindent

\noindent{\bf Claim 5:} \emph{There holds
\begin{equation}\label{e53}
\lim_{t \to 0} |x|^{-\gamma}u(t) = M \delta \ \ \ \textrm{in} \
\sigma(\mathcal{M}(\mathbb{R}^d),C_b(\mathbb{R}^d)) \, .
\end{equation}}

\smallskip

\noindent Passing to the limit in \eqref{e15} as
$\lambda=\lambda_m \to \infty $ we get
\begin{equation}\label{e51}
\| |x|^{-\gamma} u(t) \|_{1} \le M \ \ \ \textrm{for a.e.\ } t>0
\, .
\end{equation}
Estimate \eqref{e51} implies, in particular, that $|x|^{-\gamma}
u(t)$ converges, up to subsequences, to some positive finite Radon
measure $\nu$ in
$\sigma(\mathcal{M}(\mathbb{R}^d),C_c(\mathbb{R}^d))$ as $t \to
0$. In view of \eqref{eq: u-tempo-lambda-ter} we know that $ U(t)
$ converges to $ M I_{2s} = I_{2s} \ast M\delta $ e.g.~in $
L^1_{\rm loc}(\mathbb{R}^d) $ as $t \to 0$, which entails $ \nu =
M \delta $ (see the end of proof of Theorem \ref{thm:
teorema-esistenza}). We have therefore proved \eqref{e53} at least
in $\sigma(\mathcal{M}(\mathbb{R}^d),C_c(\mathbb{R}^d))$. In order
to recover convergence in
$\sigma(\mathcal{M}(\mathbb{R}^d),C_b(\mathbb{R}^d)) $, it
suffices to show that
\begin{equation*}\label{e54}
\lim_{t \to 0} \| |x|^{-\gamma}u(t) \|_1 = M \, ;
\end{equation*}
but this is a consequence of \eqref{e51} and weak$^\ast$\! lower
semi-continuity.

From Claims 1 and 5 we conclude that $ u $ solves \eqref{e10} in
the sense of Definition \ref{defsol2-barenblatt}, and therefore
coincides with $ u_M $ in view of the uniqueness result in Theorem
\ref{thm: teorema-esistenza}.
\end{proof}

We are now in position to prove Theorem \ref{thmab1}.
\begin{proof}[Proof of Theorem \ref{thmab1}]
With no loss of generality we can and shall assume that $ u_0 \in L^1_\rho(\mathbb{R}^d) \cap L^\infty(\mathbb{R}^d) $ (recall the smoothing effect \eqref{eq: smoothing-effect-general}).

Take any sequence $ \lambda_n \to \infty $. Our first aim is to
prove that, along any of the subsequences $ \{ \lambda_m\} \subset
\{ \lambda_n\} $ given by Proposition \ref{prorisc}, there holds
\begin{equation}\label{eq: first-conv-loc}
\lim_{m \to \infty} \int_{B_R} \left| u_{\lambda_m}(x,t) -
u^{c_\infty}_M(x,t) \right| |x|^{-\gamma} \,\mathrm{d}x =0 \ \ \
\forall R>0 \, , \ \forall t>0 \, .
\end{equation}
Thanks to the smoothing estimates \eqref{eq:
smoothing-effect-general}, \eqref{e18} and to the fact that for
almost every $t>0$ we know that $\{ u_{\lambda_m}(t) \}$ converges
pointwise almost everywhere to $ u^{c_\infty}_M(t) $, by dominated
convergence
\begin{equation}\label{eq: first-conv-loc-unweight}
\lim_{m \to \infty} \int_{B_R} \left| u_{\lambda_m}(x,t) -
u^{c_\infty}_M(x,t) \right| \,\mathrm{d}x =0 \ \ \ \forall R>0 \,
, \ \textrm{for a.e. } t>0 \, .
\end{equation}
Moreover, estimate \eqref{eq: stima-strong} for $u_\lambda$ reads (see Appendix \ref{rd})
\begin{equation}\label{eq: stima-strong-lambda}
\left\| {(u_\lambda)}_t(t) \right\|_{1,{\rho_\lambda}} \le
\frac{2}{(m-1) \, t} M  \ \ \ \textrm{for a.e.\ } t>0  \, .
\end{equation}
Gathering \eqref{eq: stima-strong-lambda} and \eqref{eq:
bounds-rho-pre}, we can assert that for every $R,\tau>0$ there
exists a positive constant $C(R,\tau)$ (independent of
$\lambda$) such that
\begin{equation}\label{eq: stima-strong-lambda-tau}
\left\| {(u_\lambda)}_t(t) \right\|_{L^1(B_R)} \le C(R,\tau) \ \ \
\textrm{for a.e.\ } t \ge \tau  \, .
\end{equation}
Of course \eqref{eq: stima-strong-lambda-tau} also holds for
$u^{c_\infty}_M$. It is now possible to infer that \eqref{eq:
first-conv-loc-unweight} actually holds for \emph{every} $t>0$:
\begin{equation}\label{eq: first-conv-loc-unweight-every}
\lim_{m \to \infty} \int_{B_R} \left| u_{\lambda_m}(x,t) -
u^{c_\infty}_M(x,t) \right| \,\mathrm{d}x =0 \ \ \ \forall R>0 \,
, \ \forall t>0 \, .
\end{equation}
In fact, for any given $t_0,\varepsilon >0$, there exists $t>t_0$
such that \eqref{eq: first-conv-loc-unweight} holds and
$|t-t_0|\le \varepsilon$. Exploiting \eqref{eq:
stima-strong-lambda-tau}, we get:
\begin{equation}\label{eq: first-conv-loc-unweight-every-proof}
\begin{aligned}
& \int_{B_R} \left| u_{\lambda_m}(x,t_0) - u^{c_\infty}_M(x,t_0) \right| \,\mathrm{d}x \\
\le & \int_{B_R} \left| u_{\lambda_m}(x,t_0) - u_{\lambda_n}(x,t) \right| \,\mathrm{d}x + \int_{B_R} \left| u_{\lambda_m}(x,t) - u^{c_\infty}_M(x,t) \right| \,\mathrm{d}x  + \int_{B_R} \left| u^{c_\infty}_M(x,t) - u^{c_\infty}_M(x,t_0) \right| \,\mathrm{d}x \\
\le & 2 \, C(R,t_0) \, \varepsilon + \int_{B_R} \left|
u_{\lambda_m}(x,t) - u^{c_\infty}_M(x,t) \right| \,\mathrm{d}x \, .
\end{aligned}
\end{equation}
Letting $ m \to \infty $ in \eqref{eq: first-conv-loc-unweight-every-proof} yields
\begin{equation}\label{eq: first-conv-loc-unweight-every-proof-2}
 \limsup_{m \to \infty} \int_{B_R} \left| u_{\lambda_m}(x,t_0) - u^{c_\infty}_M(x,t_0) \right| \,\mathrm{d}x \le 2 \, C(R,t_0)  \varepsilon \, .
\end{equation}
Letting now $ \varepsilon \to 0 $ in \eqref{eq:
first-conv-loc-unweight-every-proof-2} shows that \eqref{eq:
first-conv-loc-unweight} holds for $t=t_0$ as well. The validity
of \eqref{eq: first-conv-loc} is then just a consequence of
\eqref{eq: first-conv-loc-unweight-every}, the local integrability of $|x|^{-\gamma}$ and the uniform bound over $\| u_{\lambda_m}(t) - u^{c_\infty}_M(t) \|_\infty $ ensured by the smoothing estimates \eqref{eq: smoothing-effect-general} and \eqref{e18}.

The consequence of Proposition \ref{prorisc} and what we proved
above is that \emph{any} sequence $\lambda_n \to \infty $
satisfies \eqref{eq: first-conv-loc} along subsequences. We can
thus infer that
\begin{equation}\label{eq: first-conv-loc-lambda}
\lim_{\lambda \to \infty} \int_{B_R} \left| u_{\lambda}(x,t) -
u^{c_\infty}_M(x,t) \right| |x|^{-\gamma} \,\mathrm{d}x =0 \ \ \
\forall R>0 \, , \ \forall t>0 \, .
\end{equation}
Upon fixing $t=1$, relabelling $\lambda$ as $t$ and recalling the
definition of $u_\lambda$, note that \eqref{eq: first-conv-loc-lambda} reads
\[
\lim_{t \to \infty} \int_{B_R} \left| t^\alpha u(t^\kappa x, t) -
u^{c_\infty}_M(x,1) \right| |x|^{-\gamma} \,\mathrm{d}x =0 \ \ \
\forall R>0 \, .
\]
Performing the change of variable $ y=t^\kappa x $ and using the fact that $\alpha+\kappa(\gamma-d)=0 $, we obtain:
\begin{equation}\label{eq: first-conv-loc-relabel-1}
\lim_{t \to \infty} \int_{B_{Rt^\kappa}} \left| u(y,t) -
t^{-\alpha}u^{c_\infty}_M(t^{-\kappa}y,1) \right| |y|^{-\gamma}
\,\mathrm{d}y = \lim_{t \to \infty} \int_{B_{Rt^\kappa}} \left| u(y,t) -
u^{c_\infty}_M(y,t) \right| |y|^{-\gamma} \,\mathrm{d}y = 0
\end{equation}
for all $R>0$, where we used \eqref{eq: barenblatt-rescaled-identity} with $\lambda=t^{-1}$.

From now on we shall denote as $ \varepsilon_R $ any function of the spatial variable (possibly constant) which is independent of $t$ and vanishes uniformly as $R \to \infty $. Going back to the original variable $ x=t^{-\kappa} y $ we find that
\begin{equation}\label{eq: first-conv-loc-relabel-3}
\int_{B^c_{Rt^\kappa}} u^{c_\infty}_M(y,t) \,
|y|^{-\gamma} \,\mathrm{d}y = \int_{B_{R}^c}
u^{c_\infty}_M(x,1) \, |x|^{-\gamma} \,\mathrm{d}x = \varepsilon_R \ \ \ \forall R>0 \, .
\end{equation}
Hence, the conservation of mass for $u^{c_\infty}_M$, \eqref{eq: first-conv-loc-relabel-1} and \eqref{eq: first-conv-loc-relabel-3} imply that
\begin{equation}\label{eq: first-conv-loc-relabel-5}
\lim_{t \to \infty} \int_{B_{Rt^\kappa}} u(y,t) \, |y|^{-\gamma}
\,\mathrm{d}y = M c_\infty^{-1}-\varepsilon_R  \ \ \ \forall R>0
\, .
\end{equation}
Next we show that
\begin{equation}\label{eq: massa-reg-sing}
\lim_{t \to \infty} \int_{\mathbb{R}^d} u(y,t) \, |y|^{-\gamma}
\,\mathrm{d}y = M c_\infty^{-1} \, .
\end{equation}
To this end first notice that, thanks to \eqref{eq: ass-rho-slow-bis} and \eqref{eq: ass-rho-slow-lim}, there holds
\[
|y|^{-\gamma} = \frac{\rho(y)}{c_\infty + \varepsilon_R(y)} \ \ \ \forall y \in B_R^c \, ,
\]
whence
\begin{equation}\label{eq: massa-reg-sing-1}
\int_{\mathbb{R}^d} u(y,t) \, |y|^{-\gamma} \,\mathrm{d}y =
\int_{B_R} u(y,t) \, |y|^{-\gamma} \,\mathrm{d}y +
\int_{B^c_R} u(y,t) \,
\frac{\rho(y)}{c_\infty+\varepsilon_R(y)} \,\mathrm{d}y \, .
\end{equation}
Thanks to \eqref{eq: massa-reg-sing-1} and the conservation of mass \eqref{eq: cons-mass} for $u$, we get:
\begin{equation}\label{eq: massa-reg-sing-2}
\begin{aligned}
\left| \int_{\mathbb{R}^d} u(y,t) \, |y|^{-\gamma} \,\mathrm{d}y - Mc_\infty^{-1} \right| = & \left| \int_{\mathbb{R}^d} u(y,t) \, |y|^{-\gamma} \,\mathrm{d}y - \int_{\mathbb{R}^d} u(y,t) \, \frac{\rho(y)}{c_\infty} \,\mathrm{d}y \right| \\
\le &  \int_{B_R} u(y,t) \, |y|^{-\gamma} \,\mathrm{d}y \\
   & + \int_{B_R} u(y,t) \, \frac{\rho(y)}{c_\infty} \,\mathrm{d}y + \frac{\left\|\varepsilon_R\right\|_\infty}{c_\infty(c_\infty-\left\|\varepsilon_R\right\|_\infty)} \, \int_{B^c_R} u(y,t) \, \rho(y) \,\mathrm{d}y \, .
\end{aligned}
\end{equation}
Letting $t \to \infty$ in \eqref{eq: massa-reg-sing-2}, using the
smoothing effect \eqref{eq: smoothing-effect-general} (as a decay estimate) and the fact that both $\rho(y)$ and $|y|^{-\gamma}$ are locally integrable, we obtain:
\begin{equation}\label{eq: massa-reg-sing-3}
\limsup_{t \to \infty} \left| \int_{\mathbb{R}^d} u(y,t) \,
|y|^{-\gamma} \,\mathrm{d}y - M c_\infty^{-1} \right| \le \frac{M
\left\|\varepsilon_R\right\|_\infty}{c_\infty(c_\infty-\left\|\varepsilon_R\right\|_\infty)} \, .
\end{equation}
By letting $R \to \infty $ in \eqref{eq: massa-reg-sing-3} we get \eqref{eq: massa-reg-sing}. Now notice that
\begin{equation}\label{eq: first-conv-glob-1}
\begin{split}
\int_{\mathbb{R}^d} \left| u(y,t) - u^{c_\infty}_M(y,t) \right| |y|^{-\gamma} \,\mathrm{d}y \le  & \int_{B_{Rt^\kappa}} \left| u(y,t) - u^{c_\infty}_M(y,t) \right| |y|^{-\gamma} \,\mathrm{d}y \\
& + \int_{B^c_{Rt^\kappa}} u(y,t) \,
|y|^{-\gamma} \,\mathrm{d}y + \int_{B^c_{Rt^\kappa}} u^{c_\infty}_M(y,t) \, |y|^{-\gamma} \,\mathrm{d}y
\, .
\end{split}
\end{equation}
Moreover, \eqref{eq: first-conv-loc-relabel-5} and \eqref{eq: massa-reg-sing} imply that
\begin{equation}\label{eq: first-conv-glob-2}
\lim_{t \to \infty} \int_{B^c_{Rt^\kappa}}
u(y,t) \, |y|^{-\gamma} \,\mathrm{d}y = \varepsilon_R \, .
\end{equation}
Collecting \eqref{eq: first-conv-loc-relabel-1}, \eqref{eq: first-conv-loc-relabel-3}, \eqref{eq: first-conv-glob-1} and \eqref{eq: first-conv-glob-2} we finally get
\[
\limsup_{t \to \infty} \int_{\mathbb{R}^d} \left| u(y,t) -
u^{c_\infty}_M(y,t) \right| |y|^{-\gamma} \,\mathrm{d}y \le 2
\varepsilon_R \, ,
\]
whence \eqref{e104} follows by letting $R \to \infty $. The
validity of \eqref{e104-equiv} is just a consequence of
\eqref{e104} and the change of variable $y=t^\kappa x $ (one
exploits again the scaling property \eqref{eq:
barenblatt-rescaled-identity} of $u^{c_\infty}_M$).
\end{proof}

\appendix
\section{Well posedness of the parabolic problem for rapidly decaying densities} \label{rd}
Throughout this section, we shall use of the same notations as in Section \ref{subl}.

\noindent{\bf Part I.} If $\rho\in L^\infty_{\rm loc}({\mathbb R}^d)$ is positive and such that $\rho^{-1}\in L^\infty_{\rm loc}({\mathbb R}^d)$, $u_0$ is nonnegative and such that $u_0\in L^1_\rho(\mathbb{R}^d) \cap L^\infty(\mathbb{R}^d)$, then we can argue as in the proof of \cite[Theorem 7.3 (first construction)]{DQRV} in order to get the existence of a weak solution to problem \eqref{eq: regular}, in the sense of Definition \ref{defsol1}, which is bounded in the whole of $\mathbb{R}^d \times(0,\infty)$. Furthermore, the following $L^1_\rho$ comparison principle holds true:
\begin{equation}\label{e23b}
\int_{\mathbb{R}^d} \left[ u_1(x,t)- u_2(x,t) \right]_{+}  \rho(x) \mathrm{d}x \leq \int_{\mathbb{R}^d} \left[ u_{01} - u_{02} \right]_{+}  \rho(x) \mathrm{d}x \ \ \ \forall
t>0 \, ,
\end{equation}
where $u_1$ and $ u_2 $ are the solutions to problem \eqref{eq: regular}, constructed as above, corresponding to the initial data $u_{01}  \in L^1_\rho(\mathbb{R}^d) \cap L^\infty(\mathbb{R}^d)$ and $u_{02} \in L^1_\rho(\mathbb{R}^d) \cap L^\infty(\mathbb{R}^d) $, respectively.

As for uniqueness, a quite standard result for (suitable) weak solutions to problem \eqref{eq: regular} is the following.
\begin{pro}\label{Ol}
Let $\rho \in L^\infty_{\rm loc}({\mathbb R}^d)$ be positive and
such that $\rho^{-1}\in L^\infty_{\rm loc}({\mathbb R}^d)$. Let $u$ and $ v$ be two nonnegative weak solutions to \eqref{eq: regular}, corresponding to the same nonnegative $u_0 \in L^{1}_{\rho}(\mathbb{R}^d)$, in the sense that:
\begin{equation}\label{e2u-req1}
u ,v \in L^{m+1}_{\rho}(\mathbb{R}^d\times(0,\infty)) \, ,
\end{equation}
\begin{equation}\label{e2u-req2}
u^m, v^m \in L^2_{\rm loc}([0, \infty); \dot{H}^{s}(\mathbb{R}^d))
\end{equation}
and
\begin{equation}\label{e2u}
\begin{aligned}
& - \int_0^\infty \int_{\mathbb{R}^d}  u(x,t) \varphi_t(x,t) \, \rho(x) \mathrm{d}x \mathrm{d}t + \int_0^\infty \int_{\mathbb{R}^d} (-\Delta)^{\frac s 2}(u^m)(x,t) (-\Delta)^{\frac s 2} (\varphi)(x,t) \, \mathrm{d}x \mathrm{d}t \\
=& - \int_0^\infty \int_{\mathbb{R}^d}  v(x,t) \varphi_t(x,t) \, \rho(x) \mathrm{d}x \mathrm{d}t + \int_0^\infty \int_{\mathbb{R}^d} (-\Delta)^{\frac s 2}(v^m)(x,t) (-\Delta)^{\frac s 2} (\varphi)(x,t) \, \mathrm{d}x \mathrm{d}t \\
=& \int_{\mathbb R^d} u_0(x) \varphi(x,0) \, \rho(x) \dd x
\end{aligned}
\end{equation}
holds true for any $\varphi \in C^\infty_c(\mathbb{R}^d\times [0,\infty))$. Then, $u=v$ a.e.\ in $\mathbb R^d\times (0,\infty)$.
\begin{proof}
In view of the hypotheses on $u$ and $ v$, using a standard approximation argument one can show that the so-called \emph{Ole\u{\i}nik's test function}
$$ \varphi(x,t)=\int_t^{T}\left[u^m(x,\tau)-v^m(x,\tau)\right] \, \mathrm{d}\tau \ \ \textrm{in } \mathbb R^d \times (0,T] \, , \quad \varphi\equiv 0 \ \ \textrm{in } \mathbb R^d\in (T,\infty) \, , $$
is in fact an admissible test function in the weak formulations \eqref{e2u} (for each $T>0$). The conclusion then follows by arguing exactly as in \cite[Theorem 6.1]{DQRV} (see also the subsequent remark).
\end{proof}
\end{pro}

Let us discuss some further properties of the solutions we constructed, which can be proved by means of standard tools. To begin with note that, by proceeding exactly as in \cite[Lemma 8.5]{Vaz07}, one can show that $\rho \, u_t(t)$ is a Radon measure on $\mathbb{R}^d$ satisfying the inequality
\begin{equation}\label{eq: radon}
\left\| \rho \, u_t(t) \right\|_{\mathcal{M}(\mathbb{R}^d)} \le \frac{2}{(m-1) \, t} \left\| u_0 \right\|_{1,\rho} \quad \textrm{for a.e.\ } t > 0 \, ,
\end{equation}
where here, as opposed to Subsection \ref{bs}, with a slight abuse of notation we indicate by $\mathcal{M}(\mathbb{R}^d)$ the Banach space of Radon measures on $\mathbb{R}^d$ endowed with the usual norm of the total variation. Letting $z=u^{\frac{m+1}{2}}$ and following \cite[Lemma 8.1]{DQRV} we also get the validity of the estimate
\begin{equation}\label{e101}
\int_{t_1}^{t_2} \int_{\mathbb{R}^d}  \left| z_t(x,t) \right|^2  \rho(x) \mathrm{d}x \mathrm{d}t \leq C^\prime \quad \forall t_2>t_1>0
\end{equation}
for some positive constant $C^\prime$ depending on $m$, $t_1$, $t_2$ and on the initial datum. In view of \eqref{e101} and the general result provided by \cite[Theorem 1.1]{BG} one infers that $u_t \in L^1_{\rm
loc}((0,\infty);L^1_\rho(\mathbb{R}^d))$. Moreover, the inequality
\begin{equation}\label{eq: stima-strong}
\left\| u_t(t) \right\|_{1,\rho} \le \frac{2}{(m-1) \, t} \left\| u_0 \right\|_{1,\rho} \quad \textrm{for a.e.\ } t > 0
\end{equation}
holds true as a direct consequence of \eqref{eq: radon}. In particular, our solution $u$ is also a \emph{strong solution} to problem \eqref{eq: regular} in the sense of Definition \ref{defsol2}. The fact that solutions are strong permits to assert that
they also solve the differential equation in \eqref{eq: regular}, for a.e.\ $t>0$, in the $L^1$ sense. This allows to get the following energy estimate (for the details, see e.g.\ \cite[Sections 4.1 and 4.2]{GMP2}):
\begin{equation}\label{e3}
\int_{t_1}^{t_2} \int_{\mathbb{R}^d} \left|(-\Delta)^{\frac s
2}(u^m)(x,t)\right|^2 \, \mathrm{d}x \mathrm{d}t + \frac
1{m+1}\int_{\mathbb{R}^d}u^{m+1}(x,t_2) \, \rho(x) \mathrm{d}x =
\frac 1{m+1}\int_{\mathbb{R}^d} u^{m+1}(x,t_1) \, \rho(x) \mathrm{d}x \, ,
\end{equation}
for all $t_2>t_1>0 $. Furthermore, by suitably exploiting the celebrated Stroock-Varopoulos inequality (see \cite[Proposition 8.5]{DQRV} or \cite[Section 4.2]{GMP2}), one can show that for any $p \in [1,\infty] $ the $L^p_\rho$ norm of $u(t)$ does not increase in time.

Now suppose that, in addition to the above hypotheses, $\rho \in L^\infty(\mathbb{R}^d)$. Thanks to the latter assumption, from the classical fractional Sobolev embedding (we refer the reader e.g.\ to the survey paper \cite{hiker} and references quoted therein) one immediately deduces the validity of the following weighted, fractional Sobolev inequality:
\begin{equation}\label{e95}
\left\| v \right\|_{\frac{2d}{d-2s},\rho}\leq \widetilde{C}_S
\left\|(-\Delta)^{\frac s 2} (v) \right\|_{2} \ \ \ \forall v \in \dot{H}^s(\mathbb{R}^d) \, ,
\end{equation}
where $\widetilde{C}_S=\widetilde{C}_S(\| \rho \|_\infty,s,d)$ is a suitable positive constant. By interpolation it is straightforward to check that, as a consequence of \eqref{e95}, also the weighted, fractional Nash-Gagliardo-Nirenberg inequality
\begin{equation}\label{e96}
\left\| v  \right\|_{q,\rho} \leq  \widetilde{C}_{GN} \left\| (-\Delta)^{\frac
s 2} (v) \right\|_{2}^{\frac{1}{a+1}}  \left\| v \right
\|_{p,\rho}^{\frac{a}{a+1}} \ \ \ \forall v \in
L^p_{\rho}(\mathbb{R}^d) \cap \dot{H}^s(\mathbb{R}^d)
\end{equation}
holds true for any $a \ge 0$, $ p \ge 1 $ and
\[
q= \frac{2d(a+1)}{d\frac{a}p+ d-2s } \, ,
\]
where $\widetilde{C}_{GN}=\widetilde{C}_{GN}(\| \rho \|_\infty,a,p,s,d)$ is another suitable positive constant. Taking advantage of \eqref{e96}, by means of the same techniques as in \cite[Section 8.2]{DQRV} or \cite[proof of Proposition 4.6]{GMP2}, one can prove the smoothing estimate
\begin{equation}\label{e5}
\left\| u(t) \right \|_\infty \leq  K \, t^{-\alpha_p} \left\| u_0 \right \|_{p,\rho}^{\beta_p} \quad \forall t>0 \, , \ \forall p \ge 1 \, ,
\end{equation}
where
$$ \alpha_p=\frac d{d(m-1)+{2s p}} \, , \quad \beta_p=\frac{2s p \alpha_p}{d} $$
and $K=K(\| \rho \|_\infty,m,s,d)>0$.

Still under the additional assumption $\rho \in L^\infty(\mathbb{R}^d)$, it is possible to construct solutions to \eqref{eq: regular} corresponding to any nonnegative data $u_0 \in L^1_\rho(\mathbb{R}^d)$. One proceeds picking a sequence of nonnegative data $u_{0n} \in L^1_\rho(\mathbb{R}^d) \cap L^\infty(\mathbb{R}^d)$ such that $u_{0n} \to u_0 $ in $L^1_\rho(\mathbb{R}^d)$ and pass to the limit in \eqref{e2} as $n \to \infty $ by exploiting \eqref{e23b}, \eqref{e3} and \eqref{e5} for $p=1$ (see also \cite[Theorem 6.5 and Remark 6.11]{PT1}). Such solutions are still strong because the $L^1_\rho$ comparison principle \eqref{e23b} is preserved (which is in fact one of the main tools to prove that solutions are strong -- see again \cite[Section 8.1]{DQRV} and references quoted).

We have therefore proved the existence result contained in Proposition \ref{prop2c}. As concerns uniqueness, one can reason as follows. Proposition \ref{Ol}, in particular, ensures that if $ u_0 \in L^1_\rho(\mathbb{R}^d) \cap L^\infty(\mathbb{R}^d) $ then the  solution to \eqref{eq: regular} that we constructed above is unique in the class of weak solutions satisfying \eqref{e2u-req1}, \eqref{e2u-req2} and \eqref{e2u}. Moreover, any weak solution $u(x,t)$ to \eqref{eq: regular}, in the sense of Definition \ref{defsol1}, is such that $ u(x,t+\varepsilon) $ is a weak solution to \eqref{eq: regular}, corresponding to the initial datum $ u_0(x,\varepsilon) \in L^1_\rho(\mathbb{R}^d) \cap L^\infty(\mathbb{R}^d) $, satisfying \eqref{e2u-req1}, \eqref{e2u-req2} and \eqref{e2u}, for any $ \varepsilon>0 $. Thanks to these properties, one can then proceed exactly as in the proof of \cite[Theorem 6.7]{PT1}.

\smallskip

\noindent{\bf Part II.} We describe here another method for constructing weak solutions to problem \eqref{eq: regular}. Take again nonnegative initial data $u_0 \in L^1_\rho(\mathbb R^d)\cap L^\infty(\mathbb{R}^d)$ and consider the following problem (see also the discussion at the beginning of Section \ref{subl}):
\begin{equation}\label{02019}
\begin{cases}
L_s \left(\tilde{u}^m_{R}\right)= 0  & {\rm in } \ \Omega_R \times(0,\infty) \, , \\
\tilde u_{R} = 0 & {\rm on } \ \Sigma_R \times(0,\infty) \, , \\
\tilde{u}_R = u_R & {\rm on } \  \Gamma_R \times(0,\infty)  \, , \\
\displaystyle \frac{\partial \left(\tilde{u}^m_{R}\right)}{\partial y^{2s}} = \rho \displaystyle \frac{\partial u_{R}}{\partial t} & {\rm on } \  \Gamma_R \times(0,\infty) \, , \\
u_{R}= u_0 &  {\rm on } \ B_R \times \{t=0 \} \, .
\end{cases}
\end{equation}
\begin{den}\label{020110}
A weak solution to problem \eqref{02019} is a pair of nonnegative functions $(u_R,\tilde{u}_R)$ such that:
\begin{itemize}
\item{} $u_R \in C([0, \infty); L^{1}_{\rho}(B_R)) \cap L^\infty(B_R \times (\tau,\infty)) $ for all $ \tau>0 $; 
\item{} $ \tilde{u}_R^m \in L^2_{\rm loc}((0,\infty); X_0^{s}(\Omega_R))$;
\item{} $\tilde{u}_R|_{\Gamma_R \times (0,\infty)}=u_R$;
\item{} for any $\psi\in C^\infty_c((\Omega_R \cup \Gamma_R) \times (0, \infty))$ there holds
\[
-\int_0^\infty \int_{B_R} u_R(x,t) \, \psi_t(x,0,t) \, \rho(x) \dd x \dd t + {\mu_s} \int_0^\infty \int_{\Omega_R} y^{1-2s} \langle \nabla(\tilde{u}_R^m) , \nabla{\psi} \rangle(x,y,t) \, \dd x \dd y \dd t = 0 \, ;
\]
\item{} $\lim_{t\to 0} u_R(t)=u_0|_{B_R}$ in $L^1_\rho(B_R)$.
\end{itemize}
\end{den}
Weak sub- and supersolutions to \eqref{02019} are meant in agreement with Definition \ref{020110}. In addition, we say that $(u_R,\tilde{u}_R)$ is a \emph{strong solution} if $(u_R)_t \in L^\infty((\tau,\infty); L^1_{\rho}(B_R))$ for every $\tau>0$. By means of the same arguments used in the proof of \cite[Theorem 6.2]{DQRV}, it is direct to deduce the next comparison principle.
\begin{pro}\label{prop1cl} Let $\rho\in L^\infty_{\rm loc}({\mathbb R}^d)$ be positive and
such that $\rho^{-1}\in L^\infty_{\rm loc}({\mathbb R}^d)$. Let $(u^{(1)}_{R},\tilde{u}^{(1)}_R)$ and $(u^{(2)}_R,\tilde{u}^{(2)}_R)$ be a strong subsolution and a strong supersolution, respectively, to problem \eqref{02019}. Suppose that $ u^{(1)}_R \le u^{(2)}_R $ on $B_R \times \{t=0 \}$ and $\tilde{u}^{(1)}_R \le \tilde{u}^{(2)}_R $ on $\Sigma_R \times(0,\infty)$. Then $ u^{(1)}_R \le u^{(2)}_R $ in $B_R \times (0,\infty) $ and $\tilde{u}^{(1)}_R \le \tilde{u}^{(2)}_R$ in $\Omega_R \times (0,\infty)$.
\end{pro}

Making use of quite standard tools (see e.g.\ \cite{DQRV,PT1}), one can prove that for any $R>0$ and $ u_0 \in L^1_\rho(\mathbb R^d)\cap L^\infty(\mathbb{R}^d)$ there exists a unique strong solution $(u_R,\tilde{u}_R)$ to problem \eqref{02019} (in the sense of Definition \ref{020110}). Moreover, the limit function $ u=\lim_{R\to\infty} u_R$ (note that the family $ \{u_R \}$ is monotone in $R$ thanks to Proposition \ref{prop1cl}) is nonnegative, bounded in $ \mathbb{R}^d \times (0,\infty) $ and such that \eqref{e2u-req1}, \eqref{e2u-req2} and \eqref{e2u} hold true. Hence, in view of Proposition \ref{Ol}, such a $u$ necessarily coincides with the solution constructed in Part I: this in particular ensures that $ u \in C([0,\infty), L^1_\rho(\mathbb{R}^d))$. Again, for general data $u_0\in L^1_\rho(\mathbb R^d)$, we can select a sequence $\{u_{0n}\} \subset L^1_\rho(\mathbb{R}^d) \cap L^\infty(\mathbb R^d)$ such that $0\leq u_{0n} \le u_0 $ and $u_{0n} \to u_0 $ in $L^1_\rho(\mathbb{R}^d)$ and pass to the limit in \eqref{e2} as $ n \to \infty $ to get a solution to \eqref{eq: regular} in the sense of Definition \ref{defsol1} (which still coincides with the one obtained in Part I).

Finally, we should note that in \cite{DQRV} and \cite{PT1} the approximate problems are a little different from \eqref{02019} (namely, cylinders in the upper plane are used instead of half-balls). However, this change does not affect the construction of the solution $u$. Indeed, the present idea of using problem \eqref{02019} is taken from \cite[Section 2]{DQRV2}, where the case $s=1/2$ and $ \rho \equiv 1$ is studied.

\smallskip

\noindent{\bf Part III.} Let us now address the following problem, which is the analogue of \eqref{02019} in the whole upper plane:
\begin{equation}\label{eu1}
\begin{cases}
L_{s} (\tilde{u}^m) = 0 & \textrm{in } \mathbb{R}^{d+1}_+ \times (0,\infty)  \, , \\
\tilde{u}=u & \textrm{on } \partial \mathbb{R}^{d+1}_+ \times (0,\infty) \, , \\
\displaystyle \frac{\partial \left(\tilde{u}^m\right)}{\partial y^{2s}} = \rho \displaystyle \frac{\partial u}{\partial t} & \textrm{on }  \partial \mathbb{R}^{d+1}_+ \times(0,\infty) \, , \\
u = u_0 & \textrm{on } \mathbb{R}^d \times \{t=0 \} \, .
\end{cases}
\end{equation}
\begin{den}\label{defsol1u}
A nonnegative function $u$ is a local weak solution to problem \eqref{eu1} corresponding to the nonnegative initial datum $u_0 \in L^1_\rho(\mathbb{R}^d)$ if, for some nonnegative function $\tilde u$ such that
$$\tilde{u}^m \in L^2_{\rm loc}((0,\infty); X_{\rm loc}^s) \cap L^\infty(\mathbb{R}^{d+1}_+ \times (\tau,\infty)) \quad \forall \tau>0 \, , $$
there hold:
\begin{itemize}
\item $u\in C([0,\infty); L^1_{\rho}(\mathbb{R}^d))\cap L^\infty(\mathbb{R}^d\times(\tau,\infty))$ for all $\tau>0$;
\item $\tilde{u}|_{\partial\mathbb{R}^{d+1}_+ \times (0,\infty)}=u$;
\item for any $\psi\in C^\infty_c ((\mathbb R^{d+1}_+\cup \partial \mathbb R^{d+1}_+)\times (0,\infty))$,
\begin{equation}\label{e60u-ext}
-\int_0^\infty \int_{\mathbb{R}^d} u(x,t)\, \psi_t(x,0,t) \, \rho(x) \dd x \dd t + {\mu_s}  \int_0^\infty \int_{\mathbb R^{d+1}_+} y^{1-2s} \left\langle \nabla(\tilde{u}^m) , \nabla{\psi} \right\rangle(x,y,t) \, \dd x\dd y \dd t
\end{equation}
(in fact $\tilde{u}^m$ is a local extension for $u^m$);
\item for any $\varphi \in C^\infty_c(\mathbb{R}^d\times (0,\infty))$,
\begin{equation}\label{e60u}
-\int_0^\infty \int_{\mathbb{R}^d} u(x,t) \varphi_t(x,t)  \,
\rho(x) \mathrm{d}x \mathrm{d}t +  \int_0^\infty
\int_{\mathbb{R}^d} u^m(x,t) (-\Delta)^{s} (\varphi)(x,t) \,
\mathrm{d}x \mathrm{d}t  = 0 \, ;
\end{equation}
\item $ \lim_{t \to 0} u(t) = u_0$ in $L^1_{\rho}(\mathbb{R}^d) $.
\end{itemize}
Moreover, we say that $u$ is a \emph{local strong solution} if, in addition, $u_t\in
L^\infty((\tau,\infty); L^1_{\rho,\rm loc}(\mathbb{R}^d))$ for every $\tau>0$.
\end{den}
Notice that \eqref{e60u} is related to the so-called \emph{very weak} formulation of problem \eqref{eq: regular} (see also Definition \ref{defsol05} below). For local weak solutions, in general $u^m \not\in L^2_{\rm loc}((0, \infty); \dot{H}^{s}(\mathbb{R}^d))$. Hence, equivalence between \eqref{e60u-ext} and \eqref{e60u} cannot be established.

The criterion of Proposition \ref{Ol} here is not applicable in order to prove uniqueness. However, it is possible to restore the latter by imposing extra integrability conditions, as stated in Theorem \ref{prthm6}.  In order to prove it, we need some preliminaries. Given a nonnegative $f\in  C^\infty_c(\mathbb{R}^d)$, let $ h= I_{2s} \ast f $, so that
\begin{equation}\label{e61}
(-\Delta)^s (h) =  f \quad \textrm{in} \ \mathbb{R}^d \,.
\end{equation}
It is not difficult to show that $ h \in C^\infty(\mathbb{R}^d)$, $ h \ge 0 $ and
\[
h(x) + \left|\nabla h(x) \right| \leq K \, |x|^{-d+2s} \quad \forall x
\in \mathbb{R}^d
\]
for some $K>0$ (use e.g.~Lemma \ref{l100}). Now take a cut-off function $\xi \in  C^\infty_c(\mathbb{R}^d)$
such that $0 \le \xi \le 1$ in $\mathbb{R}^d$, $\xi = 1 $ in $B_{1/2}$ and
$\xi = 0 $ in $B_{1}^c$. For any $R>0$, let
\begin{equation}\label{e82v}
\xi_R(x)= \xi\left(\frac{x}{R}\right) \quad \forall x \in
\mathbb{R}^d \, .
\end{equation}
After straightforward computations, we obtain:
\[
(-\Delta)^s\left(h \xi_R \right)(x)= h(x)(-\Delta)^{s}(\xi_R)(x) +
(-\Delta)^s (h)(x) \, \xi_R(x) + \mathcal{B}(h,\xi_R)(x) \ \ \
\forall x \in \mathbb{R}^d \, ,
\]
where $\mathcal{B}(\phi_1,\phi_2)(x)$ is the bilinear form defined as
\[
\mathcal B(\phi_1,\phi_2)(x) = 2 \, C_{d,s}
\int_{\mathbb{R}^d}\frac{(\phi_1(x)-\phi_1(y))(\phi_2(x)-\phi_2(y))}{|x-y|^{d+2s}} \,
\mathrm{d}y \quad \forall x \in \mathbb{R}^d
\]
and $C_{d,s}$ is the positive constant appearing in \eqref{eq: def-basilare-frac-lap}. In view of \cite[Lemma 3.1]{PV}, we have the following crucial result.
\begin{lem}\label{lemma1v}
Let $f\in  C^\infty_c(\mathbb{R}^d)$, with $f \geq 0$, $ h= I_{2s}
\ast f $ and $\xi_R$ be as in \eqref{e82v}. Then, for any $T>0$ and
$v \in L^1_{(1+|x|)^{-d+2s}}(\mathbb{R}^d\times (0,T))$, there
holds
\[
\lim_{R \to \infty} \int_0^T \int_{\mathbb{R}^d} \left| v(x,t) \,
h(x) \, (-\Delta)^s (\xi_R)(x)\right| \, \mathrm{d}x \mathrm{d}t +
\int_0^T \int_{\mathbb{R}^d} \left| v(x,t) \,
\mathcal{B}(h,\xi_R)(x)\right| \, \mathrm{d}x \mathrm{d} t = 0 \, .
\]
\end{lem}
We are now in position to prove Theorem \ref{prthm6}.
\begin{proof}[Proof of Theorem \ref{prthm6}] 
Let $\underline{u}$ be the weak solution to problem \eqref{eq: regular} provided by Proposition \ref{prop2c}.
Its minimality in the class of solutions described by Definition
\ref{defsol1u}, namely local strong solutions, is a consequence of
the construction outlined above and the comparison principle given
in Proposition \ref{prop1cl}: the approximate solutions whose
limit is $ \underline{u} $, let them be $ u_R $ or $ u_n $, are
smaller than any local strong solution. As concerns estimate
\eqref{e10a-decay}, we only mention that it can be established by
means of the same arguments as in the proof of \cite[Theorem
5.5]{PT1}, combined with the smoothing effect \eqref{e5} (see also
\cite[Remark 6.11]{PT1}). We point out that the approximate
solutions $ u_R $ used there are those obtained by solving
\eqref{02019} in cylinders $ \mathcal{C}_R $ rather than in
half-balls $ \Omega_R $. Nevertheless, by comparison, the solution
to the Dirichlet problem in $ \Omega_R $ is below the one in $
\mathcal{C}_R $, and this is clearly enough to get the upper bound
\eqref{e10a-decay}. If $ u_0 \in L^\infty(\mathbb R^d) $ no
smoothing effect is needed, so that the validity of
\eqref{e10a-decay} down to $t_0=0$ is ensured by the fact that $
\underline{u} \in L^\infty(\mathbb R^d\times(0,\infty)) $ (see
again \cite[Theorem 5.5]{PT1}).

In order to prove our uniqueness
results, let us first assume that $ u_0 \in L^1_\rho(\mathbb R^d)
\cap L^\infty(\mathbb R^d) $. In this case, we have just shown
that
\begin{equation*}\label{eq: estim-riesz-0}
\int_{0}^{T} \underline{u}^m(x,\tau) \, \mathrm{d}\tau \le K
(I_{2s} \ast \rho)(x)
\end{equation*}
for all $T>0$. In view of Lemma \ref{l100}, it is straightforward
to check that
 $ I_{2s} \ast \rho \in L^1_{(1+|x|)^{-d+2s}}(\mathbb R^d) $ provided $ d > 4s $ and $ \gamma > 4s $, whence
 $ \underline{u}^m \in L^1_{(1+|x|)^{-d+2s}}(\mathbb R^d \times (0,T)) $ for all $ T>0 $ and such values of the parameters.
 Moreover, since $ \underline{u} \in L^1_\rho(\mathbb R^d \times (0,T)) \cap L^\infty(\mathbb R^d \times (0,T)) $ and so $ \underline{u}^m \in L^1_\rho(\mathbb R^d \times (0,T)) $,
 if $ \gamma \in (2s,d-2s] $ we have again that $ \underline{u}^m \in L^1_{(1+|x|)^{-d+2s}}(\mathbb R^d \times (0,T))
 $.

Now take another local strong
solution $u$ to \eqref{eq: regular} corresponding to the same $
u_0 \in L^1_\rho(\mathbb R^d) \cap L^\infty(\mathbb R^d) $, which
for the moment we assume to be bounded as well in the whole of $
\mathbb R^d \times (0,\infty) $. Because both $u$ and
$\underline u$ belong to
$$ C([0,\infty); L^1_{\rho}(\mathbb{R}^d)) \cap L^\infty(\mathbb{R}^d \times \mathbb (0,\infty)) \, , $$
by exploiting \eqref{e60u} it is direct to see that for any
$\varphi \in C^\infty_c(\mathbb{R}^d\times [0,\infty))$ there
holds
\begin{equation}\label{e60}
\begin{aligned}
& - \int_0^\infty \int_{\mathbb{R}^d} u(x,t) \varphi_t(x,t) \, \rho(x) \mathrm{d}x \mathrm{d}t +  \int_0^\infty \int_{\mathbb{R}^d} u^m(x,t) (-\Delta)^{s} (\varphi)(x,t) \, \mathrm{d}x \mathrm{d}t \\
= & - \int_0^\infty \int_{\mathbb{R}^d} \underline{u}(x,t) \varphi_t(x,t)  \, \rho(x) \mathrm{d}x \mathrm{d}t +  \int_0^\infty \int_{\mathbb{R}^d} \underline{u}^m(x,t) (-\Delta)^{s} (\varphi)(x,t) \, \mathrm{d}x \mathrm{d}t \\
= & \int_{\mathbb{R}^d} u_0(x) \varphi(x,0) \, \rho(x) \mathrm{d}x
\, .
\end{aligned}
\end{equation}
Let $\eta\in C^\infty([0,+\infty))$ be such that
\[
\eta(0)= 1 \, , \quad \eta \equiv 0 \ \ \textrm{in} \ [1,+\infty)
\, , \quad 0 < \eta < 1 \ \ \textrm{in} \ (0,1) \, , \quad
\eta^\prime \leq 0 \ \ \textrm{in} \ [0,+\infty) \, .
\]
For any $ T>0 $, set
\[
\eta_{T}(t) := \eta(t/T) \ \ \ \forall t\ge 0 \, .
\]
Take the test function
\[
\varphi(x,t)= h(x) \xi_R(x) \eta_{T}(t) \quad \forall(x,t) \in
\mathbb{R}^d \times [0,+\infty) \, ,
\]
where $ h $ and $ \xi_R $ are defined above, and plug it in the
weak formulation \eqref{e60} solved by $u-\underline{u}$. We get:
\begin{equation}\label{eq: diff-u-umin}
\begin{aligned}
& \int_0^{T} \int_{\mathbb{R}^d} f(x) \xi_R(x) \eta_{T}(t) \, [u^m(x,t)-\underline{u}^m(x,t)] \, \mathrm{d}x \mathrm{d}t \\
= & \int_0^{T}\int_{\mathbb{R}^d} h(x) \xi_R(x) \eta^\prime_{T}(t) \, [u(x,t)-\underline u(x,t)] \, \rho(x) \mathrm{d}x \mathrm{d}t \\
 & - \int_0^{T} \int_{\mathbb{R}^d} \left[ h(x)(-\Delta)^s(\xi_R)(x) + \mathcal{B}(h,\xi_R)(x) \right] \eta_{T}(t) \,[u^m(x,t)-\underline{u}^m(x,t)] \, \mathrm{d}x \mathrm{d}t \, .
\end{aligned}
\end{equation}
Since $\underline{u} \le u $ and $\eta^\prime_{T}\le 0$, from
\eqref{eq: diff-u-umin} we find that
\begin{equation}\label{eq: diff-u-umin-bis}
\begin{aligned}
0 \le & \int_0^{T}\int_{\mathbb{R}^d} f(x) \xi_R(x) \eta_{T}(t) \, [u^m(x,t)-\underline{u}^m(x,t)] \,\mathrm{d}x \mathrm{d}t \\
\le & \int_0^{T} \int_{\mathbb{R}^d} \left|
h(x)(-\Delta)^s(\xi_R)(x) + \mathcal{B}(h,\xi_R)(x) \right|
[u^m(x,t)+\underline{u}^m(x,t)] \, \mathrm{d}x \mathrm{d}t \, .
\end{aligned}
\end{equation}
Letting $R\to \infty$ in \eqref{eq: diff-u-umin-bis} and applying
Lemma \ref{lemma1v} with $v=u^m+\underline{u}^m$, we then infer
that $u=\underline{u}$ in the region $ \textrm{supp}\, f \times
(0,T)$. Thanks to the arbitrariness of $f$
 and $T$, this means that $u = \underline{u}$ in the whole of $\mathbb{R}^d \times (0,\infty)$.

We finally need to get rid of the
assumption $ u_0 \in L^{\infty}(\mathbb R^d) $. First notice that,
for any $ t_0>0 $, $ u $
 and $ \underline{u} $, restricted to $ \mathbb R^d \times(t_0,\infty) $, are \emph{bounded} local strong solutions with initial
 data $ u(t_0) \in L^1_\rho(\mathbb R^d) \cap L^\infty(\mathbb R^d) $ and $ \underline{u}(t_0) \in L^1_\rho(\mathbb R^d) \cap L^\infty(\mathbb R^d) $,
 respectively. Moreover, $ u^m,\underline{u}^m \in L^1_{(1+|x|)^{-d+2s}}(\mathbb R^d \times (t_0,T)) $ for all $ T> t_0 $.
 Hence, in view of the uniqueness result we proved above, they both coincide with the corresponding minimal local strong
 solutions having the same initial data. But minimal solutions are, in fact, the ones constructed above, for which, in particular, the $ L^1_\rho $ comparison principle \eqref{e23b} holds true.
 As a consequence, $ \| u(t)-\underline u(t) \|_{1,\rho} \le \| u(t_0)-\underline u(t_0) \|_{1,\rho} $ for all $ t>t_0>0 $.
 The conclusion follows by letting $ t_0 \to 0 $ and recalling that $ u,\underline{u} \in C([0,\infty); L^1_{\rho}(\mathbb{R}^d)) $.
\end{proof}

Let us consider the next definition of {\em very weak} solution to problem \eqref{eq: regular}.
\begin{den}\label{defsol05}
A nonnegative function $u\in L^\infty(\mathbb{R}^d\times
(0,\infty))$ is a very weak solution to problem \eqref{eq:
regular} corresponding to the nonnegative initial datum $u_0 \in L^\infty(\mathbb{R}^d)$ if, for any $\varphi \in C^\infty_c(\mathbb{R}^d\times [0,\infty))$,
\eqref{e60} holds true.
\end{den}

Clearly, any \emph{bounded} weak solution to \eqref{eq: regular} (according to Definition \ref{defsol1}) is also a very weak solution in the sense of Definition \ref{defsol05}.

\begin{oss}\label{oss101}\rm
As a byproduct of the method of proof of the uniqueness result in
Theorem \ref{prthm6}, it turns out that if $u_1$ and $u_2$ are
\emph{ordered very weak} solutions to problem to \eqref{eq:
regular} (i.e.~$ u_1 \le u_2 $ or $ u_2 \le u_1 $ in $
\mathbb R^d \times (0,\infty) $) such that $ u_1^m,u_2^m \in
L^1_{(1+|x|)^{-d+2s}}(\mathbb R^d \times (0,T)) $ for all $ T>0 $,
then $u_1 = u_2$.
\end{oss}

\bibliographystyle{plainnat}


\end{document}